\documentclass{amsart}
\usepackage{color}
\usepackage{amsthm,amssymb,verbatim}
\usepackage{graphicx}
\usepackage{enumerate}
\usepackage{cite}
\newcommand{\Rr}{\mathcal R}
 \newcommand{\Cc}{\mathcal C}
 \newcommand{\Dd}{\mathcal{D}}
 
 \newcommand{\Ll}{\mathcal{L}}
  \newcommand{\LL}{\mathcal{L}}

  \newcommand{\MM}{\mathcal{M}}
  \newcommand{\GG}{\mathcal{G}}

 \newcommand{\Bb}{\mathcal B}
 \newcommand{\RR}{\mathbf{R}}  % reals
 \newcommand{\BB}{\mathbf{B}}  %ball
  \newcommand{\HH}{\mathbf{H}}  %hyperbolic space
  \newcommand{\ppm}{{\{p^+,p^-\}}}  
  \newcommand{\Div}{\operatorname{Div}}
  \newcommand{\ddiv}{\operatorname{div}}
    \newcommand{\dist}{\operatorname{dist}}
 
 \newcommand{\area}{\operatorname{area}}
 \newcommand{\eps}{\epsilon}

 \newcommand{\UU}{\mathcal{U}}

\input epsf
\def\begfig {
\begin{figure}
\small }
\def\endfig {
\normalsize
\end{figure}
}

\swapnumbers

    \newtheorem{theorem}    {Theorem}       [section]
    \newtheorem{lemma}      [theorem]       {Lemma}
    \newtheorem{corollary}  [theorem]     {Corollary}
    \newtheorem{proposition}       [theorem]       {Proposition}
    
    \newtheorem*{claim}{Claim}
    \newtheorem*{theorem*}{Theorem}
    \theoremstyle{definition}
    \newtheorem{definition}  [theorem] {Definition}

    \theoremstyle{definition}
    \newtheorem{remark}   [theorem]       {Remark}

  \newcommand{\marginpor}[1]{}
\title{Limiting behavior of  sequences of properly embedded minimal disks}

\author{David Hoffman}
\address{Department of Mathematics\\ Stanford University\\ Stanford, CA 94305}
\email{dhoffman@stanford.edu}
\author{Brian White}
\address{Department of Mathematics\\ Stanford University\\ Stanford, CA 94305}
\thanks{The research of the second author was supported by NSF grant DMS~1404282,
and by a grant from the Simons Foundation.}
\email{bcwhite@stanford.edu}
\date{May 31, 2017}

\usepackage{hyperref}
\usepackage[alphabetic, msc-links, backrefs]{amsrefs}

\begin{document}
\maketitle
\begin{abstract}
We develop a theory of ``minimal $\theta$-graphs'' and characterize the  behavior of limit laminations of such surfaces, including an understanding of their limit leaves and their curvature blow-up sets.
We use this to prove  that it is possible to realize families of catenoids in euclidean space as limit leaves of sequences of embedded minimal disks, even when there is no curvature blow-up.  Our methods work in a more general Riemannian setting, including hyperbolic space. This allows us  to establish the existence of a complete, simply connected,  minimal surface in hyperbolic space that is not properly embedded.
\end{abstract}
\section{Introduction}
Let $D_n$ be a sequence of properly embedded minimal disks in an open subset $W$ of a Riemannian $3$-manifold.
Then there is a subsequence $D_{n(i)}$ such that   
 the curvatures of the $D_{n(i)}$ blow up at the points of closed subset $K\subset W$ (possibly empty), and such that the 
$D_{n(i)}$ converge smoothly away from $K$ to a minimal lamination $\LL$ of $W\setminus K$.   
One would like to know what closed sets $K$ and what laminations $\LL$ can arise in this way.  
Colding and Minicozzi proved very strong theorems about such $K$
and $\LL$.  In particular, they showed (under mild hypotheses on the ambient metric)
that $K$ is contained in a rectifiable curve,
and that for each point $p$ in $K$, there is a unique leaf $L$ of the lamination 
such that that $p\in \overline{L}$
and such that $L\cup\{p\}$ is smooth.
(See~\cite{CM4}*{Section I.1}. See also~\cite{CM4}*{Theorem 0.1} for a
closely related result.)
Later it was shown that $K$ is contained in a $C^1$ curve, and that $L\cup\{p\}$ is perpendicular to that curve.  See~\cite{MeeksRegularity1} and~\cite{white-C1}.

 In this paper, we give
a more detailed description of the lamination and of the singular set
for a certain rich class of minimal disks.   In particular, we prove

\begin{theorem}\label{qualExistence}
Let $\BB\subset \RR ^3$ be the unit ball and let $Z\subset \RR ^3$ be the vertical coordinate axis. Suppose $D_n$ is a sequence of properly embedded minimal disks in the ball $\BB$
with the property  that each disk $D=D_n$ satisfies
\begin{equation}\label{extthetagraph}
\text{$\BB\cap Z\subset D$, and the images of $D\setminus Z$
under rotations about $Z$ foliate $\BB\setminus Z$.}
\end{equation}

\noindent Then there is a subsequence $D_{n(i)}$, a relatively closed subset $K$ of $\BB\cap Z$,
and a minimal lamination $\Ll$ of $\BB\setminus K$ such that
\begin{enumerate}
\item[1.] The curvatures of the $D_{n(i)}$ blow-up precisely at the points of $K$.
\item[2.] The $D_{n(i)}$ converge smoothly away from $K$ to the lamination $\Ll$.
\item[3.] The limit leaves of $\Ll$ are catenoids and rotationally invariant disks.
\item[4.] The curvature blow-up set $K$ is precisely the set of centers of the disks in statement 3.
\item[5.]  If $L$ is a non-limit leaf of $\Ll$, then $L\setminus Z$ and its rotations around $Z$ foliate
an open subset of $\BB\setminus Z$. In fact, each component of the complement of the limit leaves of $\Ll$ in $\BB\setminus Z$ is foliated by the rotations of such a non-limit leaf.
\end{enumerate}
\end{theorem}

It is not hard to produce examples of disks $D_n$ satisfying condition~\eqref{extthetagraph} of Theorem~\ref{qualExistence}.
In particular, let $C_n \subset \partial \BB$ be a smooth, simple closed curve that intersects
each horizontal circle in $\partial \BB$ in exactly two diametrically opposite points.
   Then there is unique embedded minimal disk $D_n$ such that $\partial D_n=C_n$
and such that $Z\cap \BB\subset D_n$.  Furthermore, it is easy to show that the disk satisfies condition \eqref{extthetagraph} of Theorem~\ref{qualExistence}. 
 See Section~\ref{ExistenceUniqueness} below.  

By choosing suitable curves $C_n$ and taking the corresponding disks $D_n$,
 we can produce interesting examples of blow-up sets $K$ and 
limit laminations $\Ll$.   
For example, let $\MM$ be the lamination of $\BB$ consisting of all  the area-minimizing catenoids in $\BB$ with axis $Z$
that are symmetric
about the $xy$-plane, together with all horizontal disks that are disjoint from those catenoids (See Figure~\ref{AreaMinimizingCatenoidsAndDisks}.)
We show  that there is a sequence $D_n$ of properly embedded minimal disks in $\BB$  with a  limit lamination $\Ll$ (from Theorem~\ref{qualExistence})  whose 
rotationally invariant leaves are precisely the surfaces in $\MM$ and that has exactly one  
leaf that is not rotationally invariant. (That additional leaf contains a segment of $Z$.)
More generally, if $\MM^*$ is essentially any symmetric sublamination of $\MM$, we show that there is a sequence
$D_n$ such that the rotationally invariant leaves of the limit lamination are precisely 
the surfaces in $\MM^*$. 
(See Theorem~\ref{realizing-M(T)}.) Of course by 
Statement~4 of Theorem~\ref{qualExistence},
the curvatures of the $D_{n(i)}$ blow-up precisely at the centers of the disks in $\MM^*$.
\begfig
\includegraphics[width=2.05in]{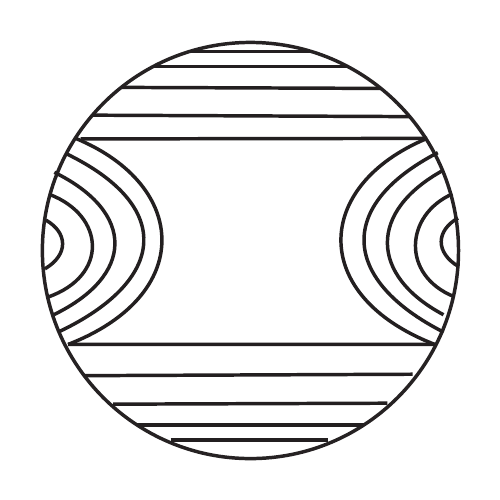}
 \hfill
 \begin{center}
 \parbox{4.1in}{ 
\caption{\label{AreaMinimizingCatenoidsAndDisks} {\bf Limit Leaves $\MM$ in the unit ball $\BB$}.
Depicted here in cross section is the lamination  $\MM$ of $\BB$ consisting of all area-minimizing catenoids with
axis $Z$  and symmetry plane $\{z=0\}$, together with all horizontal disks that are disjoint from the catenoids.  Essentially
any symmetric sublamination $\MM^*$ of $\MM$ can be realized as the set of limit leaves of a limit lamination of a sequence of properly embedded minimal disks in $\BB$. This is proved in Theorem~\ref{realizing-M(T)}.
}
 }
 \end{center}
 \endfig

The results stated above remain true if the Euclidean metric on $\BB$ is replaced by 
  any suitable rotationally symmetric Riemannian metric.  
   In particular, they remain true for the Poincar{\'e} metric
on $\BB$.   We show that many kinds of limit laminations and blow-up sets occur 
for sequences of disks that are properly embedded in all of hyperbolic space.   This is in very
sharp contrast to the situation in $\RR^3$. Consider a sequence $\BB_1\subset \BB_2 \subset \dots$
of balls that exhaust $\RR^3$ and properly embedded minimal disks $D_n\subset \BB_n$.  
By work of Colding and Minicozzi \cite{CM3},
with extensions by Meeks-Rosenberg~\cite{meeks-rosenberg-uniqueness}
and Meeks~\cite{MeeksRegularity1},
there are only three possible behaviors (after passing to a subsequence):
\begin{enumerate}[\upshape $\bullet$]
\item The $D_n$ converge smoothly to a helicoid.
\item The $D_n$ converge smoothly to a lamination of $\RR^3$ by parallel planes.
\item The curvature blow-up set $K$ is a straight line, and 
disks $D_n$ converge smoothly in $\RR^3\setminus K$ 
to the foliation consisting of all planes perpendicular to $K$.
\end{enumerate}

Note that if $D_n$ is the portion in the ball $\BB$
of a helicoid with axis $Z$  and if the curvatures of the $D_n$ tend to infinity,
then the curvature blow-up set is $Z\cap \BB$.   Colding and Minicozzi \cite{CM-proper-nonproper} constructed an example in which the blow-up set
is $Z^-\cap \BB$ (where $Z^-$ is the set of points $(0,0,z)$ with $z\le 0$.)
Khan \cite{khan} then showed that $K$ can be any finite subset of $Z\cap\BB$.  The  authors of this paper proved that $K$ can be any relatively closed
subset of $Z\cap\BB$ \cite{hoffman-white-sequences}. In particular, sets with non-integral Hausdorff dimension can arise as blow-up sets.  (Subsequently, Kleene \cite{kleene} gave another proof of this theorem.)  In all of those examples, the limit leaves of the limit lamination
are precisely the horizontal disks centered at points of $K$.  (Indeed, in all of those examples, the disks $D_n$
satisfy  conditon \eqref{extthetagraph} of Theorem~\ref{qualExistence}, and they have an additional property: the tangent plane to $D_n$ is not vertical except
at points on $Z$.) In Section~\ref{thetasurfaces} we will develop the theory of embedded minimal disks, satisfying condition~\eqref{extthetagraph}.

%Kluge with counter%
\setcounter{theorem}{+2}
%Kluge with counter%
\subsection{The mathematical advances in this paper}
\begin{enumerate}

\item[1.] We prove that it is possible to realize families of catenoids (as well as horizontal disks)
as limit leaves of a limit lamination of embedded minimal disks, even when there is no curvature blow-up. This result raises the question of whether it is possible to produce limit leaves (of a limit lamination of a sequence of embedded minimal disks) that are neither disks nor annuli. Under the assumption that $W$ is mean convex and contains no closed minimal surfaces, 
 Bernstein and Tinaglia \cite{bernstein-tinaglia} have recently proved that  the answer is no.
\item[2.]   The constructions  to produce these examples work
for more general Riemannian metrics (such as the Poincar{\'e} metric) on $\BB$.
\item[3.]  Colding and Minicozzi \cite{CM-CY} proved a general Calabi-Yau conjecture for complete embedded minimal surfaces in $\RR^3$ of finite topology: such a surface must be properly embedded. We use our limit lamination theory to prove that such a theorem fails in hyperbolic three-space, even for simply connected minimal surfaces. This was originally proved by Baris Coskunuzer \cite{BarisC} 
 by entirely different methods. Our approach yields an example on either side of any area-minimizing catenoid in hyperbolic space. See Theorem~\ref{CCCounterexample}.
\end{enumerate}

%Kluge with counter%
\setcounter{theorem}{+3}
%Kluge with counter%
\subsection{An outline of the sections of this paper}
In Section~\ref{thetasurfaces},  minimal $\theta$-graphs are introduced and their limiting behavior is analyzed. They are essentially  the surfaces satisfying condition \eqref{extthetagraph} of 
Theorem~\ref{qualExistence}. but in a more general Riemannian setting.  

In Section~\ref{ExistenceUniqueness} we prove the existence of minimal $\theta$-graphs
 with prescribed boundary.  
 In Section~\ref{SmoothConvergenceBoundary},
  we prove (under suitable hypotheses) smooth
   convergence at the boundary for sequences of minimal $\theta$-graphs. 
In Section~\ref{necessaryconditions}
we use a standard calibration-type argument
 to establish a necessary area-minimization property for laminations consisting of planes and catenoids to appear as 
the limit leaves of a limit lamination of minimal $\theta$-graphs.  
We conjecture that it is a sufficient condition.

  In Section~\ref{SpecifyingLeaves}, we use this existence results of the previous two sections to show that we can, under certain conditions, 
  specify the limit leaves of a limit laminations coming from 
  a sequence of minimal $\theta$-disks. In particular, we construct  sequences of embedded minimal disks whose limit laminations have prescribed limit-leaf sublaminations containing catenoids.  
  
In Sections~\ref{hyperbolic1}-\ref{hyperbolic3}, 
we extend the results of Sections~\ref{ExistenceUniqueness}-\ref{SpecifyingLeaves} 
to hyperbolic three-space.
Handling the infinite-area minimal surfaces that arise there requires
an additional argument. That argument (in Section~\ref{necessaryconditionshyperbolic})
 was inspired by the work of Collin and Rosenberg \cite{collin-rosenberg-harmonic} on minimal
graphs in $H^2\times\RR$.
In Section~\ref{hyperbolic3},  we prove (Theorem~\ref{CCCounterexample}) that there exists a complete and simply connected embedded minimal surface in hyperbolic space that is not properly embedded.

\section{$\theta$-graphs}\label{theta}\label{thetasurfaces}
In this section we will denote by $W$ a connected open set in $\RR^3$ that is 
rotationally symmetric  about the $x_3$-axis $Z$.

\begin{definition}{\bf ($\theta$-graph, spanning $\theta$-graph )} \label{ThetaGraphDefinition}
Let $M$ be a smooth surface  in $W\setminus Z$. Then $M$ is a {\bf  $\theta$-graph} if it can be written in the form \begin{equation}\label{form}
  \{ (r\cos\theta(r,z), r\sin\theta(r,z), z):  (r,z)\in V\},
\end{equation}
where $\theta (r,z):V\rightarrow\RR$ is a smooth, real-valued function, and $V$ is an open subset
of $$\{(r,z):\, r>0,\,\, (r,0,z)\in W\}.$$
\end{definition}
\noindent

 A $\theta$-graph  $M\subset W\setminus Z$ intersects 
 each rotationally invariant circle at most once. 
 We say that $M$ is a {\bf spanning $\theta$-graph} if it intersects every rotationally
 invariant circle in $W\setminus Z$ exactly once. This is equivalent to  the assertion that the domain of definition of $\theta(r,z)$ equals $\{(r,z):\, r>0,\,\, (r,0,z)\in W\} $, and also equivalent to  the requirement that 
 $M$ and its rotated images foliate $W\setminus Z$.

\begin{remark}[Simple examples of $\theta$-graphs] \label{SimpleExamples} 
Let $W=\RR^3$ and 
$V=\{(r,z):\, r>0\}$. 
 If we let  $\theta (r,z) =c$ in~\eqref{form},
then the surface is a vertical halfplane with boundary $Z$.
If we let  $\theta (r,z) =z/\alpha$, for any $\alpha\neq0$, then the surface
is a half-helicoid with pitch $2 \pi \alpha$ and axis $Z$.
\end{remark}
\noindent
 
\begin{lemma}\label{thetagraphs} Let $W\subset \RR ^3$ be a rotationally invariant domain.  Suppose $M$ is a smooth  embedded   surface in $W\setminus Z$. Then the following  two conditions are equivalent:  
\begin{enumerate}
\item[1.]
\begin{enumerate}
		\item  Given any rotationally invariant circle $S$, either $M$ is disjoint from $S$ or intersects $S$ precisely once, and the intersection is transverse.
		\item Any closed curve in $M$ has winding number $0$ about $Z$.
\end{enumerate}
\item[2.] 
 $M$ is a $\theta$-graph.
\end{enumerate}
\end{lemma}
\begin{proof} 
Statement $1(a)$ is equivalent to a weakened form of  Statement~2, produced  by replacing the function
 $\theta:V\rightarrow \RR$ in Definition~\ref{ThetaGraphDefinition}  by  a  smooth function taking values in $\RR$ modulo $2\pi$. The function $\theta$ lifts to a single-valued function into $\RR$  if and only if assertion~$1(b)$ holds.
\end{proof}

\stepcounter{theorem}

 \subsection{The relationship between  
  spanning $\theta$-graphs and the surfaces of Theorem~\ref{qualExistence} }
  We are interested in properly embedded minimal surfaces
    $M\subset W$ with $\partial M\subset \partial W$ 
that satisfy the following property:
\begin{equation}
\text{$W\cap Z\subset M$, and the rotations of $M\setminus Z$ 
      foliate $W\setminus Z$.}
 \label{extended}
\end{equation} 
This is condition \eqref{extthetagraph} of Theorem~\ref{qualExistence} stated for the domain $W$. 
As indicated in the introduction, all the existence theorems for sequences of minimal disks 
are proved by producing surfaces of this kind.
 Their intimate relationship with  spanning $\theta$-graphs is given by the following lemma.

\begin{lemma}
\label{extthetagraphs} Let $W\subset \RR ^3$ be a  simply connected  domain that is rotationally symmetric around $Z$. Let $M$ be a smooth,  properly embedded surface in $W$.  Then
$M$ satisfies \eqref{extended} if and only if $M\setminus Z$ consists of two components, 
each of which is a  spanning $\theta$-graph, and the components are related by $\rho_Z$, 
 $180^\circ$ rotation about $Z$ by $\pi$.
  \end{lemma}
 
\begin{proof} The lemma follows immediately from the definitions.
\end{proof}

We will focus on  spanning $\theta$-graphs in this paper,
 mindful that Lemma~\ref{extthetagraphs}
provides the link between these graphs and  their doubles, the surfaces  of
  Theorem~\ref{qualExistence} and its generalization,
   Theorem~\ref{laminationproperties} below.

\stepcounter{theorem}

\subsection{$\theta$-graphs considered as graphs in the simply connected covering of $W\setminus Z$}\label{ThetaGraphGraphs}

%%%%%%%%

%%%%%%%%%%%

 We will have occasion in Section~\ref{SmoothConvergenceBoundary}  and in the Appendix to view  $\theta$-graphs as surfaces lying
 in the simply connected covering of 
 $W\setminus Z$. 
Suppose we have a domain
 \begin{equation} 
 V\subset\{(r,z):\, r>0,\,\, (r,0,z)\in W\}.
 \end{equation}
   For $p=(r,z)\in V$, and $\theta\in \RR$, let  $\pi:V\times \RR\to W\setminus Z$ be the mapping $\pi(p,\theta)=(r\cos\theta, r\sin\theta,z)$. Note that a rotation around $Z$ in $W$ corresponds to a vertical translation in $V\times\RR$.    In this setting, the definition of a $\theta$-graph $M$ 
     (Definition~ \ref{ThetaGraphDefinition})   is equivalent to the following: 
 \begin{itemize}
 \item[]   {\it The surface $M$ can be lifted to $V\times\RR$  as a graph of a smooth function $\theta:V\to\RR$. }
 \end{itemize}

 Suppose now that $W$ is endowed with a rotationally invariant metric $g$. Pulling back  $g$  to $V\times\RR$ produces a metric $g^*$ on $V\times\RR$ in which vertical translations  are isometries (corresponding to rotations in $W$). Note that the metric $g^*$ on $V\times \RR$ is not the product metric. 
A surface $M$ is $g$-minimal in $W\setminus Z$ if and only if  its lift $M^*$ is $g^*$-minimal in $V\times\RR$.

The  simple examples  in Remark~\ref{SimpleExamples} with  $W=\RR^3$  and $V=\{(r,z):\, r>0,\,\}$ are minimal surfaces. They lift to minimal surfaces $V\times\RR$:
The vertical halfplane in $\RR^3 $ bounded by $Z$ lifts to a horizontal planar slice  $\theta(r,z) =c$;
the  half-helicoid in $\RR^3$ with axis $Z$ lifts to the graph of $\theta (r,z) =z/\alpha$, $\alpha\neq 0$, 
a   halfplane that is neither vertical nor horizontal.  

\begin{theorem}[Boundary regularity theorem for minimal $\theta$-graphs]\label{theta-boundary-theorem}
Suppose that $M\subset W\setminus Z$ is a spanning $\theta$-graph that is minimal for a smooth,
rotationally invariant metric on $W$.  Then 
  $M\cup (W\cap Z)=\overline{M}\cap W$
is a smooth manifold-with-boundary, the boundary being $Z\cap W$.

Now suppose that $W$ is bounded and simply connected and that the metric extends smoothly
to $\overline{W}$.
Let $\Gamma=\overline{M}\cap \partial W$.  If $\Gamma\cup\rho_Z\Gamma$ is a smooth,
simple closed curve, then $\overline{M\cup\rho_ZM}$ 
is a smooth, embedded manifold-with-boundary.
\end{theorem}

The first assertion is local, so it suffices to consider the case when $W$ is a simply connected,
which implies that $M$ is a disk.
(Otherwise, replace $W$ and $M$ by $\BB(p,r)\subset W$ and $M\cap\BB(p,r)$, 
where $p\in Z\cap W$.)  

Thus Theorem~\ref{theta-boundary-theorem}
 is an immediate consequence of the following more general boundary regularity theorem:

\begin{theorem}\cite{white-newboundary}.
Suppose that $U$ is an open subset of a smooth Riemannian $3$-manifold,
that $C$ is a smooth,  properly embedded curve in $U$, that $D$ is a 
properly embedded minimal surface in $U\setminus C$, and that $D\cup C$ is topologically a manifold
with boundary, the boundary being $C$.  Then $D\cup C$ is a smooth manifold-with-boundary.
\end{theorem}

%cheating on counters%
%\setcounter{theorem}{+7}
%cheating on counters%
\stepcounter{theorem}
\subsection{Properties of 
 limit laminations of sequences of minimal spanning $\theta$-graphs}
\label{limits of theta graphs} 
We now state and prove the main theorem of this section.

\begin{theorem}\label{laminationproperties}
Suppose that the open unit ball $\BB$ in $\RR^3$ is 
 endowed with a smooth Riemannian metric that is rotationally invariant around $Z$. 
Suppose that $D_n$ is a sequence of minimal spanning $\theta$-graphs
  in $\BB\setminus Z$.
Then, after passing to a subsequence, the
$D_n$  converge smoothly on compact subsets of $\BB\setminus Z$
 to a minimal lamination $\Ll$ of $\BB$ with the following properties:
\begin{enumerate}[\upshape \qquad 1.]
% 1
\item\label{main-theorem-item1} Each leaf of $\Ll$ is either rotationally symmetric about $Z$
 or is a $\theta$-graph.
 %  2
\item\label{main-theorem-item2} Each rotationally invariant circle in $\BB$ either is contained
in a rotationally invariant leaf of $\Ll$ or else intersects $\Ll$ transversely in a single point.
%  3
\item\label{main-theorem-item3} 
The limit leaves of $\Ll$ are precisely the leaves that are rotationally invariant
 about $Z$.
\end{enumerate}
Let $\LL'$ be the set of rotationally invariant leaves in $\LL$, and let $K$
be the set of points in $\BB\cap Z\cap \overline{\cup\LL'}$.
\begin{enumerate}[\upshape \qquad 1.]
\setcounter{enumi}{3}
%  4
\item\label{main-theorem-item4}  Each connected 
component $\mathcal{O}$ of $\BB\setminus\overline{\cup \LL'}$
 contains a unique leaf $L$ of $\Ll$.
 That leaf is a spanning $\theta$-graph in $\mathcal{O}$, and 
  $\mathcal{O}$ contains no other points of $\LL$.
  Furthermore, $\overline{L}\cap\mathcal{O}$ is a smooth manifold-with-boundary (the
 boundary being $Z\cap \mathcal{O}$), and $\overline{D_n}\cap \mathcal{O}$ converges 
  to $\overline{L}\cap \mathcal{O}$
 smoothly on compact subsets of $\mathcal{O}$.
 %  5
 \item\label{main-theorem-schwarz-item} Each component $I$ of $(\BB\cap Z) \setminus K$ lies on 
the boundary of 
a non-limit leaf $L \in\Ll$.  The leaf $L$ can be
 extended smoothly by Schwarz reflection across $Z$, 
 and no point in $I$ lies in the closure of  $\Ll\setminus L$.
%   6
 \item\label{main-theorem-puncture-item}
 The lamination $\Ll'$ extends smoothly to a lamination of $\BB$.
 If $p\in K$, then there is a unique leaf $L(p)\in \Ll'$ whose closure contains $p$, and 
 $L(p)\cup\{p\}$ is a smooth surface that meets $Z$ orthogonally.
%   7
\item\label{main-theorem-blowup-item}   The curvature blowup  of the $D_n$ 
occurs precisely at the points of  $K$. 
 \end{enumerate}
\end{theorem}

In the following corollary (and throughout the paper), $\rho_Z$ denotes $180^\circ$ 
rotation about $Z$.

\begin{corollary}
The doubled disks $D_n\cup (\BB\cap Z)\cup \rho_ZD_n$
converge smoothly in $\BB\setminus K$ to the lamination $\LL^*$ 
obtained from $\LL$ as follows: for each connected component $\mathcal{O}$
of $\BB\setminus\overline{\cup \LL'}$, we replace the leaf $L$ in $\mathcal{O}$ 
(see Statement~\ref{main-theorem-item4})
by  $\overline{L}\cup\rho_Z\overline{L}\cap \mathcal{O}$.   
In particular, $\LL$ and $\LL^*$ have the same rotationally invariant leaves.
\end{corollary}

\begin{proof}[Proof of theorem] 
The rotational Killing field $\partial /\partial\theta$ defines a Jacobi field $J_n$ on each $D_n$.
Note that $D_n$ is stable because  $J_n$ never vanishes.  (In fact $D_n$ has 
a certain area-minimizing property: see Corollary~\ref{existence-corollary}.)
Thus the curvature is uniformly bounded on compact subsets of $\BB\setminus Z$,
so a subsequence converges smoothly to a lamination $\Ll$.  In particular, there is no
curvature blowup in $\BB\setminus Z$.
 By
relabeling, we may assume that the subsequence is the original sequence.

Let $L$ be a leaf of $\Ll$.  As above, there  is a Jacobi field $J$ on $L$, defined by the rotational vector field $\partial/\partial \theta$. This Jacobi field  does
not change sign on $L$ since $J_n$ does not vanish on $D_n$.
Thus by the maximum principle, it either vanishes nowhere on $L$
or it vanishes everywhere on $L$.  In the first case, $L$ is transverse to every circle $S$ that
is rotationally invariant about $Z$. In the second case, $L$ is rotationally invariant
about $Z$.  

Let $S$ be a rotationally invariant circle in $\BB$.
Since $S$ is compact and since it intersects each $D_n$, it must also intersect $\LL$.
 Using the previous paragraph, 
 we conclude that $S$ either intersects $\Ll$ transversally, 
 or it lies entirely in a rotationally invariant leaf of $\Ll$.
If the circle $S$ intersects $\Ll$ transversely, then it intersects $\Ll$ in a single point
since it intersects each $D_n$ in a single point.  Thus the leaf $L$ through that
point is not a limit leaf.   Let $U$ be the union of $L$ and its rotated images.
The convergence of $D_n\cap U$ to $L\cap U$ is smooth and single-sheeted,
so any closed curve  $\alpha\subset L$ is a limit of closed curves $\alpha_n$ in $D_n$.  
By Lemma~\ref{thetagraphs} (Statement~$1(b)$), 
the winding number of $\alpha_n$ about $Z$ is $0$.  
Thus the winding number of $\alpha$ about $Z$ is also $0$.
By Lemma~\ref{thetagraphs},  $L$ is a $\theta$-graph.

We have proved Statements~\ref{main-theorem-item1}
and~\ref{main-theorem-item2}, 
and we have established that limit leaves are rotationally invariant.  
To prove Statement~\ref{main-theorem-item3}, 
we must establish that  rotationally invariant leaves of $\Ll$ are limit leaves.
Suppose that $L$ is  a rotationally invariant leaf of $\Ll$, and let $p_n$ be a sequence of points in $\BB\setminus (Z\cup L)$ converging
to a point in $L$.  The rotationally invariant circle though $p_n$ contains
a point $q_n$ of the the lamination $\Ll$.  Since $p_n$ is not in $L$, neither is $q_n$.
  By passing to subsequence, we may assume
that the $q_n$ converge to a point $q\in L$.   We have shown that $L$ contains a point $q$
that is a limit of points $q_n$ in $\LL\setminus L$.  Thus $L$ is a limit leaf.

To prove Statement~\ref{main-theorem-item4},
let $\mathcal{O}$ be a connected component of $\BB\setminus \overline{\cup \Ll'}$.
By Statements~\ref{main-theorem-item1}, \ref{main-theorem-item2}, and 
\ref{main-theorem-item3}, for each point $(x,z)$ in
\[
     U:= \{ (x,z): x>0, \, (x,0,z)\in \mathcal{O}  \},
\]
the rotationally invariant circle through $(x,0,z)$ intersects the lamination in 
a single point $F(x,z)$, and $F$ defines a smooth embedding of $U$ into $\mathcal{O}$.
Since $\mathcal{O}$ is connected, $U$ is connected, and therefore $L=F(U)$
is connected.  In particular, $L$ is a leaf of $\LL$ rather than a union of leaves.
By Statement~\ref{main-theorem-item1}, $L$ is a $\theta$-graph.
We have already seen that it intersects each rotationally invariant circle in $\mathcal{O}$.
Thus $L$ is a spanning $\theta$-graph in $\mathcal{O}$.
By Theorem~\ref{theta-boundary-theorem}, $\overline{D_n}\cap \mathcal{O}$ and 
$\overline{L}\cap\mathcal{O}$ are smooth manifolds-with-boundary,
the boundary being $Z\cap\mathcal{O}$.
This proves Statement~\ref{main-theorem-item4}, except
for the assertion about smooth convergence.

We already know smooth convergence away from $Z$, so to prove the smooth
convergence in Statement~\ref{main-theorem-item4}, it suffices to consider
the case when $\BB\cap Z$ is nonempty.  
 In that case, the smooth convergence  is an 
immediate consequence of the following general theorem (which is true in arbitrary 
dimensions and codimensions):

\begin{theorem}\label{white-allard-bdry-theorem}\cite{white-controlling}*{Theorem~6.1}
Suppose that $M$ is a smooth, connected manifold-with-boundary properly embedded in
an open subset $\mathcal{O}$ of a smooth Riemannian manifold, and suppose that 
of $\mathcal{O}\cap \partial M$ is nonempty.  Suppose that $M_n$ is a sequence of smooth minimal manifolds-with-boundary
that are
properly embedded in $\mathcal{O}$ and suppose 
that $\mathcal{O}\cap \partial M_n$ converges smoothly to 
$\mathcal{O}\cap \partial M$.
Suppose also that
\begin{equation}\label{eq:set-convergence}
  \{p\in \mathcal{O}: \liminf\dist(p,M_n)=0\} \subset M.
\end{equation}
Then $M_n$ converges smoothly to $M$ on compact subsets of 
$\mathcal{O}$.
\end{theorem}

To apply Theorem~\ref{white-allard-bdry-theorem}, we let
 $M:=\overline{L}\cap \mathcal{O}$ and 
$M_n:= \overline{D_n}\cap \mathcal{O}$.  Then $\partial M=\partial M_n=Z\cap\mathcal{O}$,
and~\eqref{eq:set-convergence}
 holds because (in our situation) $M_n$ converges smoothly to $M$ on compact
subsets of $\mathcal{O}\setminus Z$.  Thus the smooth convergence 
asserted by Theorem~\ref{white-allard-bdry-theorem} holds. This completes the 
proof of Statement~\ref{main-theorem-item4}.

Statement~\ref{main-theorem-schwarz-item} follows immediately
from Statement~\ref{main-theorem-item4} by letting $\mathcal{O}$
be the connected component of $\BB\setminus \overline{\cup\LL'}$ containing the interval $I$.

We now prove Statement~\ref{main-theorem-puncture-item}.
Let $p\in K$.  By defintion of $K$, there is a sequence $p_n\in \cup\Ll'$ converging
to $p$.
Let $\alpha_n$ be the angle that the tangent plane to $\Ll'$ at $p_n$ makes
with the horizontal. 
To prove Statement~\ref{main-theorem-puncture-item}, it suffices to show that $\alpha_n\to 0$.
Let $q_n$ be the point in $Z$ nearest to $p_n$.
Translate the limit leaf through $p_n$ by $-q_n$ and dilate by $1/|p_n-q_n|$ to get
a surface $\Sigma_n$.  Note that $\Sigma_n$ is rotationally invariant and stable.  Since
it is stable, the norm of the second fundamental form times distance to $Z$ is uniformly bounded.
Thus (after passing to a subsequence) the $\Sigma_n$ converge smoothly on compact subsets
of $\RR^3\setminus Z$ to a stable, rotationally invariant minimal surface $\Sigma$.  The only
rotationally invariant minimal surfaces in $\RR^3$ are catenoids and horizontal planes.
Since catenoids are unstable, $\Sigma$ must be a horizontal plane -- in fact, the plane $z=0$.
Since this limit is independent of choice of subsequence, in fact the sequence $\Sigma_n$ converges
to the plane $z=0$.  Hence $\alpha(p_n)\to 0$, proving Statement~\ref{main-theorem-puncture-item},
except for uniqueness.

If uniqueness failed, we would have two rotationally invariant disks tangent to each other
at a point $p$ on $Z$.  The intersection set would consist of $p$ together with a collection
of rotationally invariant circles. But near a common point of two distinct minimal surfaces in a $3$-manifold, the intersection set consists of two or more curves that meet at the point.
This proves uniqueness.

 \begin{remark}\label{any-stable-remark}
The proof of Statement~\ref{main-theorem-puncture-item} shows 
 that if $L\subset \BB$ 
is a stable, rotationally invariant, embedded minimal surface that contains $p\in \BB\cap Z$
in its closure, then $L\cup\{p\}$ is a smooth minimal surface.
\end{remark}

We now prove Statement~\ref{main-theorem-blowup-item}.
By the smooth convergence $D_n\to \LL$ in $\BB\setminus Z$ and by 
Statement~\ref{main-theorem-item4},  we already know that the curvatures of the $D_n$ are uniformly bounded on compact
subsets of $\BB\setminus K$.   
Thus we need only show if $p\in K$, then the curvatures of the $D_n$ blow up at $p$.
Suppose not. Then (by passing to a subsequence) we can assume that the
curvatures of the $\overline{D_i}$ are uniformly bounded in some neighborhood of $p$.
Since the tangent plane to $\overline{D_i}$ at $p$ is vertical, it follows that for
a sufficiently small ball $\BB(p,r)\subset \BB$, the slopes of the tangent planes to  the surfaces
$D_i\cap \BB(p,r)$ are all $\ge 1$.
   Hence if $L$ is leaf of $\LL$, then the slope of the tangent 
planes to $L\cap \BB(p,r)$ are all $\ge 1$.
But by Statement~\ref{main-theorem-puncture-item}, 
since $p\in K$, there is a rotationally invariant leaf $L(p)$ such that $L(p)\cup p$
is a smooth manifold.  In particular, the tangent plane at $p$ is horizontal, so $L(p)$ contains
points arbitrarily close to $p$ with slopes arbitrarily close to $0$.
The contradiction proves Statement~\ref{main-theorem-blowup-item}, 
and thereby completes the proof of the Theorem~\ref{laminationproperties}.
 \end{proof}

 %%%%

 %%%%

\begin{proposition}\label{proper-lift-proposition}
Each leaf of $\LL$ lifts to a properly embedded surface in the universal cover $U$ of
$\BB\setminus Z$.
\end{proposition}

\begin{proof}
Let $V=\{(r,z): (r,0,z)\in \BB\}$.  Then we can regard $U=V\times\RR$ as the univeral
cover of $\BB\setminus Z$, the covering map being
\begin{align*}
  &\pi: V\times \RR \to \BB\setminus Z, \\
  &\pi(r,z,\theta) = (r\cos\theta, r\sin\theta, z).
\end{align*}
Let $L$ be a leaf of $\LL$ and let $p\in L$.  Let $p_n\in D_n$ converge to $p$.
Let $\tilde D_n$ be a lift of $D_n$ to the universal cover of $\BB\setminus Z$,
and let $\tilde p_n$ be the point in $\tilde D_n$ that projects to $p_n$.
By making suitable vertical translations, we can assume that the points $\tilde p_n$
converge to a point $\tilde p$ that projects to $p$.

Since $\tilde D_n$ is a minimal graph, it satisfies the following bound:
if $C$ is any compact region with smooth boundary in $U$, then
\begin{equation}\label{eq:area-bound}
   \area(\tilde D_n\cap C) \le \frac12 \area(\partial C).
\end{equation}
Since the $\tilde D_n$ are stable minimal surfaces, a subsequence
converges smoothly to a limit $\tilde D$.  By~\eqref{eq:area-bound}, the limit $\tilde D$ 
is properly embedded.  Note that $\tilde D$ is a lift of $L$.
\end{proof}

\begin{corollary}\label{properness-corollary}
If $\Sigma$ is a rotationally invariant leaf of $\LL$, then 
  $\Sigma$ is properly embedded in $\BB\setminus Z$.
\end{corollary}

\begin{proof}
Let 
\[
  \sigma = \{((x^2+y^2)^{1/2},z): (x,y,z)\in \Sigma\}.
\]
Then $\sigma\times \RR$ is the lift of $\Sigma$ to the universal cover.

Since $\sigma\times\RR$ is a properly embedded surface in $V\times\RR$
 (by Proposition~\ref{proper-lift-proposition}),
$\sigma$ is a properly embedded curve in $V$.  The result follows immediately.
\end{proof}

Next we prove that each rotationally invariant leaf in Theorem~\ref{laminationproperties} is either
a punctured disk or an annulus, and that the corresponding disk or annulus is properly embedded
in $\BB$.   The reader may wish to skip the proof, since the 
theorem is obviously true in the cases we are most interested in (namely,
when the Riemannian metric on $\BB$ is the Euclidean metric or the Poincar\'e metric).

\begin{proposition}\label{additional-properties-proposition}
Let $\Sigma$ be a rotationally invariant leaf in the lamination
   $\LL$.
Then either $\Sigma$ is a punctured disk such that $\overline{\Sigma}\cap \BB$ properly embedded in $\BB$,
or $\Sigma$ is an annulus that is properly embedded
in $\BB$.

Now suppose that  $\overline{\BB}$ is compact with smooth boundary,
that the metric extends smoothly to $\overline{\BB}$, and that $\overline{\BB}$
is strictly mean convex with respect to the metric.
Then $\Sigma$ is smooth at the boundary: 
$\overline{\Sigma}$ is either a smoothly embedded closed disk or
a smoothly embedded closed annulus.
\end{proposition}

\begin{proof}
Let $\mathcal{D}$ be the planar domain
\[
  \mathcal{D} = \{ (r,0,z)\in \BB: r>0\}
\]
and let $\sigma$ be the curve in $\mathcal{D}$ given by
\[
  \sigma = \Sigma\cap \mathcal{D}
\]
Thus $\Sigma$ is the surface of revolution obtained by rotating $\sigma$ around $Z$.
 
 In Corollary~\ref{no-closed-surfaces-corollary}, we show that 
\begin{equation}\label{eq:not-closed}
\text{$\overline{\Sigma}$ cannot be a smooth closed surface in $\BB$.}
\end{equation}
Thus $\sigma$ is not a closed curve, so it has two ends.

If one end of $\sigma$ contained a point $p$ of $Z\cap \BB$ and if the other end
contained a point $q$ of $Z\cap \BB$ in its closure,
   then $\Sigma\cup \{p,q\}$ would be a smooth embedded surface in $\BB$ 
(by Statement~\ref{main-theorem-puncture-item} of Theorem~\ref{laminationproperties}),
contradicting~\eqref{eq:not-closed} above.

Thus either $\sigma$ contains no points of $Z\cap \BB$ in its closure,
or exactly one end of $\sigma$ contains a point $p$ of $Z\cap \BB$ in its closure.
In the first case, $\Sigma$ is a properly embedded annulus in $\BB$.
In the second case, $\overline{\Sigma}\cap \BB = \Sigma\cup \{p\}$  
is a properly embedded disk in $\BB$.
(The properness follows from Corollary~\ref{properness-corollary} above.)

Now suppose that $\partial \BB$ is smooth and that the metric extends smoothly
to $\overline{\BB}$.  If $\sigma$ contained an endpoint $p$ of $Z\cap \BB$
in its closure, then
             $\Sigma\cup\{p\}$ would be a smooth minimal surface
(by Remark~\ref{any-stable-remark}),
contradicting the mean convexity of $\partial \BB$ at $p$.
Thus $\sigma$ cannot contain an endpoint of $Z\cap \BB$ in its closure.
It follows that at least one end of $\sigma$ contains a point $q$ of $(\partial \Dd)\setminus Z$
in its closure.  By the strict mean convexity, that end of $\sigma$ must converge to $q$.
Thus the union of $\Sigma$
and the circle corresponding to $q$ is a smooth manifold with
boundary.  
The two ends of $\sigma$ cannot converge to the same
point in $(\partial \Dd)\setminus Z$, since then $\overline{\Sigma}$ would be
a closed surface in $\overline{\BB}$, which is impossible by Corollary~\ref{annoying-corollary}.

We have shown that either $\sigma$ has one endpoint in $(\partial \Dd)\setminus Z$
and the other endpoint in $Z\cap \BB$, in which case $\overline{\Sigma}$ is a disk, or
$\overline{\Sigma}$ has both endpoints in $(\partial \Dd)\setminus Z$, in which case
$\overline{\Sigma}$ is an annulus.
\end{proof}

\section{Existence of minimal $\theta$-graphs with prescribed boundary} \label{ExistenceUniqueness}

In this section, we prove existence and uniqueness of spanning minimal $\theta$-graphs
for a large family of prescribed boundary curves.

\begin{theorem}\label{UniqueEmbeddedDisk}
Let $\BB$ be the open unit ball in $\RR^3$, and 
suppose that $\overline{\BB}$
is mean convex with respect to a smooth Riemannian metric $g$
that is rotationally invariant about $Z$.

Let $\gamma$ be a smooth curve in $\partial \BB$ 
    joining $p^+=(0,0,1)$ to $p^-=(0,0,-1)$
such that $\gamma$ intersects each horizontal circle in $\partial \BB$ exactly once,
and such that the curve $\gamma\cup \rho_Z\gamma$ is smooth.
Let $\Gamma$ be the union of $\gamma$ with $Z\cap \BB$.
 Then among all oriented surfaces (of arbitrary genus) with boundary $\Gamma$, 
 there is a unique surface $D$ of least area.
 The surface $D$ is a $\theta$-graph, and 
   $\overline{D\cup \rho_ZD}$ is a smoothly embedded disk with boundary 
   $\gamma\cup \rho_Z\gamma$.
 
 Furthermore, if $M\subset \BB$ is any oriented, embedded minimal surface
 with finite area, finite genus, and with boundary $\Gamma$, then $M=D$.
\end{theorem}

Concerning the hypothesis that $\gamma\cup \rho_Z\gamma$ is smooth,
note that smoothness of $\gamma$ implies smoothness of $\gamma\cup\rho_Z\gamma$ except
possibly at the endpoints of $\gamma$.  
For $\gamma\cup \rho_Z\gamma$ to be smooth at an endpoint of $\gamma$,
the necessary and sufficient condition is the vanishing of curvature and all even order 
derivatives
of curvature at that endpoint.

\begin{corollary}\label{existence-corollary}
Suppose that $D$ is a spanning $\theta$-graph in $\BB$ that is minimal
with respect to a smooth, rotationally invariant Riemannian metric on $\BB$.

If $\overline{\BB'}\subset \BB$ is rotationally invariant about $Z$,
mean convex, and smoothly diffeomorphic to a closed ball, 
then $D\cap\overline{\BB'}$ 
is the unique least-area integral current among all integral currents in $\overline{\BB'}$ 
having boundary $\partial (D\cap \BB')$.

If $\BB$ can be exhausted by such subdomains $\overline{\BB_n'}$,
then $D$ is an area-minimizing integral current.
\end{corollary}

\begin{proof}[Proof of corollary]
Apply the theorem to $D\cap U$.
\end{proof}

\begin{proof}[Proof of Theorem~\ref{UniqueEmbeddedDisk}]
Let $D$ be an  oriented area-minimizing surface (i.e., integral current) in $\BB$ bounded by
  $\Gamma$.   (To be precise, we let $D$ be the set of points in $\BB\setminus Z$ in the 
   support of that integral current.)
Note we are not restricting the genus of $D$.
By the Hardt-Simon boundary regularity theorem~\cite{hardt-simon}, 
  $\overline{D}$ is a smooth, embedded
manifold-with-boundary except at the corners $p^+$ and $p^-$ 
of $\Gamma$.
Let $C$ be a tangent cone to $D$ at $p^-$.  Then $C$ lies in the upper halfspace $\{x_3\ge 0\}$,
and the boundary of $C$ consists of the positive $x_3$-axis together with a horizontal ray, both
with multiplicity $1$.   The only such cone is the corresponding quarter-plane with multiplicity one.
Now $\overline{D\cup \rho_ZD}$  is a minimal surface 
with boundary $\gamma\cup\rho_Z\gamma$,  and it is smoothly immersed everywhere except possibly at $p^+$ and at $p^-$.
We have just shown that the tangent cone to 
               $\overline{D\cup \rho_ZD}$
 at $p^-$ is a halfplane with multiplicity one.
By Allard's Boundary Regularity Theorem~\cite{allard-boundary}, 
   $\overline{D\cup \rho_ZD}$
 is a smoothly embedded manifold-with-boundary near $p^-$.
Likewise, it is a smoothly embedded manifold-with-boundary near $p^+$.

Let $\sigma$ be a closed curve in $D$.  
By pushing $\sigma$ slightly in the direction of the unit normal to $D$, we get a closed curve
$\sigma'$ that is homotopic to $\sigma$ in $W\setminus Z$.  Note that $\sigma'$ is disjoint from $D$.
Thus its algebraic intersection number with $D$ is $0$.
By elementary topology, the winding number about $Z$ of a closed curve in $W\setminus Z$
is equal to its linking number with $\Gamma$, which is equal to its intersection number with $D$.
Thus the winding number of $\sigma'$ about $Z$ is $0$.
Since $\sigma$ and $\sigma'$ are homotopic in $W\setminus Z$, the same is true of $\sigma$.

We have shown: every closed curve in $D$ has winding number $0$ about $Z$.
Thus $D$ lifts to the universal cover of $W\setminus Z$.
Equivalently, there is an angle function $\theta_D: D\setminus Z\to \RR$ such that
\[
   p = ( r(p)\cos\theta_D(p), r(p)\sin\theta_D(p), z)
\]
for all $p=(x,y,z)\in D\setminus Z$, where
$
   r(p) = \sqrt{x^2+y^2}
$.   
The smoothness of $\overline{D\cup \rho_ZD}$ implies that
 $\theta_D(\cdot)$
extends continuously to $\overline{D}$.

Now define
\[
  \omega: \overline{D} \to \RR
\]
by letting $\omega(p)$ be the maximum of $\theta_D(q)-\theta_D(p)$
among all $q\in \overline{D}$ such that $q$ and $p$ lie on the same rotationally invariant circle.
Note that $\omega$ is upper semicontinuous and that $\omega=0$ on $\partial D=\Gamma$
(by the smoothness of $D$ at the boundary).
Thus if $\omega$ did not vanish everywhere, it would attain its maximum at some interior point 
  $p\in D$.
But at that point, the strong maximum principle would be violated.
(Note that the surface $D$ and the surface obtained by rotating $D$ through angle $-\omega(p)$
would touch each other at $p$.)  Thus $\omega(\cdot)\equiv 0$, which implies that $D$
is a $\theta$-graph.  Every rotationally invariant circle in $W$ links $\Gamma$
and therefore must intersect $D$.
Thus $D$ is a spanning $\theta$-graph.

To prove the uniqueness assertion, suppose that $M$ is
a finite-genus, finite-area, orientable, embedded minimal surface in $W$ with boundary $\Gamma$. 
By classical boundary regularity theory, $\overline{M\cup\rho_ZM}$ is a 
  minimal immersed
surface, possibly with branch points.  Since the boundary of 
$\overline{M\cup\rho_ZM}$
 lies on $\partial W$, it
cannot have any boundary branch points.   Also, it cannot have interior branch points in 
$W\setminus Z$ since $M$ is embedded.   Finally, it cannot have
a branch point on $Z$, since then $M$ would have a boundary branch point on $Z$, 
which implies that $M$ is not embedded near that point, a contradiction.
We have shown that $\overline{M\cup\rho_ZM}$ is a smoothly immersed surface-with-boundary.

Just as for $D$, it follows that there is a continuous angle function
\[
 \theta_M: \overline{M}\to \RR
\]
such that
\begin{equation}\label{eq:defining-property}
  p = (r(p) \cos\theta_M(p), r(p) \sin\theta_M(p), z)
\end{equation}
for $p=(x,y,z)\in \overline{M}$, where $r(p)=\sqrt{x^2+y^2}$.

Note that on $\Gamma\cap\partial W$, $\theta_M$ and $\theta_D$ 
differ by a constant multiple of $2\pi$.
Note also that adding a multiple of $2\pi$ to $\theta_M(\cdot)$
 does not affect~\eqref{eq:defining-property}.
Thus we can assume that $\theta_D\equiv\theta_M$ on $\Gamma\cap\partial W$.

Now define a continuous function 
\begin{align*}
&\phi: \overline{M} \to \RR, \\
&\phi(p) = \theta_D(q)-\theta_M(p),
\end{align*}
where $q$ is the unique point of intersection of $D$ with the rotationally
invariant circle containing $p$.
(Here we allow circles of radius $0$, so if that $p\in Z$, then $q(p)=p$.)

Now $\phi\equiv 0$ on $\overline{M}\cap \partial W$,
so if it were not everywhere $0$, then $|\phi(\cdot)|$ would
attain a strictly positive maximum at some point $p\in \overline{M}\cap W$.
But that would violate the strong maximum principle (if $p\in M$) or the strong
boundary maximum principle (if $p\in Z$).  

(If this is not clear, consider $\overline{M}$ and the surface obtained
by rotating $\overline{D}$ by angle $-\phi(p)$.  The two surfaces are tangent at $p$,
and there is a neighborhood of $p$ in which the two surfaces have no transverse intersections.)
\end{proof}

\section{Smooth convergence  at the boundary  }
\label{SmoothConvergenceBoundary}

In this section, we will assume that
\begin{enumerate}[\upshape (i)]
\item\label{boundary-form-item}
$D_n$ is a sequence of  spanning
 minimal $\theta$-graphs in $\BB\setminus Z$  with boundaries of the form 
\begin{equation*}
\partial D_n =\gamma _n \cup \overline I, 
\end{equation*}
where  $I={\BB}\cap Z$ and $\gamma_n$ is an
 embedded
curve in  $\partial \BB$ connecting 
the endpoints  
$p^+=(0,0,1)$ and $p^-=(0,0,-1)$ of $I$.
\item\label{metric-item} the Riemannian metric on $\BB$ extends smoothly to 
 $\overline{\BB}$, and  $\overline{\BB}$  is strictly mean convex.
\end{enumerate}
 In  Theorem~\ref{laminationproperties} of Section~\ref{thetasurfaces},   
 we proved that, away from a closed subset $K\subset I$, a subsequence of the $D_n$ converge smoothly to a limit lamination 
 $\Ll$.  The set $K$ is precisely the set on which the curvature of the surfaces $D_n$ blow up.
  In this section we provide conditions under which the convergence is smooth up to the boundary 
in $\partial \BB\setminus \{p^+,p^-\}$. This involves establishing uniform curvature estimates in a neighborhood of points on the boundary of $\BB$.

In Theorem~\ref{curvatureestimate0} in Appendix~\ref{MinimalGraphs}, we prove  the following  curvature estimate.

\begin{theorem}\label{curvatureestimate01}
 Suppose in addition to~\eqref{boundary-form-item} and~\eqref{metric-item}
  that
  the curvature and the first derivative of curvature of
   $\gamma _n$ are bounded independently of $n$.
  Then  the curvatures $B(D_n, \cdot)$ are uniformly  bounded away from $\overline{I}$.
\end{theorem}

This uniform curvature estimate is enough to conclude that the boundaries  of the leaves of a limit lamination  $\Ll$ are regular at the points of their boundary in $\partial \BB\setminus \ppm$. We already know from 
Theorem~\ref{laminationproperties} that they are regular at the points of $I\subset\partial D_n$ that are not in the curvature blowup set $K$.

 Theorem~\ref{curvatureestimate0}  of Appendix~\ref{MinimalGraphs} is stated in terms of minimal graphs in $V\times\RR$, where $V=\{(x,0,z)\,:\,x>0, \, (x,0,z)\in \BB\}$.  As explained in that Appendix and in Section~\ref{ThetaGraphGraphs}, this is equivalent to the situation considered in Theorem~\ref{curvatureestimate01} above.

Note that if the curves $\gamma_n$ converge smoothly to a lamination of 
 $\partial \BB\setminus \ppm$, then (away from $I$) we have uniform bounds on the curvature
 of $\gamma_n$ and the first derivative of curvature.  
 Therefore we can use Theorem~\ref{curvatureestimate01} to conclude 
 smooth convergence up to the boundary:

\begin{theorem} \label{smoothboundarybehavior}
Suppose, in addition to~\eqref{boundary-form-item} and~\eqref{metric-item},
that  the $D_n$ converge smoothly in $\BB\setminus Z$ to a lamination
 $\LL$, and that the curves $\gamma_n$ converge smoothly to a lamination $\GG$ of 
  $\partial\BB\setminus \ppm$.
  Then the convergence $D_n\to\LL$ is smooth
  up to $\partial\BB\setminus \ppm$.
  
In particular, if 
 $L$ is a leaf of $\Ll$, and if $L^*$ is a lift of it to the universal cover $\mathcal{U}$ of 
$\overline{\BB} \setminus Z$,   then the closure $\overline{L^*}$ of $L^*$ in 
$\mathcal{U}$ is a smooth embedded
manifold-with-boundary, and each component of $\partial \overline{L^*}$ projects
to a leaf $\gamma$ of $\mathcal{G}$.  Furthermore, every leaf $\gamma$ of $\mathcal{G}$
arises in this way: if $\gamma\in\mathcal{G}$, there is a lift $L^*$ of a leaf of $\Ll$ 
and a component of $\partial \overline{L^*}$ that projects to $\gamma$.
\end{theorem}

\begin{corollary}\label{smoothboundarycorollary}
If $\gamma$ is a rotationally invariant leaf of $\GG$, 
then $\LL$ contains a rotationally invariant leaf with $\gamma$ as one of its boundary
components.
\end{corollary}

\begin{proof}
Let $L$ be a leaf of $\Ll$ associated to $\gamma$ as in
 Theorem~\ref{smoothboundarybehavior}.
 (That is, suppose $L$ and $\gamma$ have lifts $L^*$ and $\gamma^*$
 to the universal cover of $\overline{\BB}\setminus Z$ such that $L^*\cup \gamma^*$
 is a smooth manifold-with-boundary.)
If $L$ is rotationally invariant, we are done.  If not, $L$ and its images
under rotations about $Z$ foliate a rotationally invariant region $\Omega$
in $\BB$.  Note that $\Omega$ is bounded by rotationally invariant leaves of $\Ll$.
Two of those leaves must each have $\gamma$ as a boundary component.
\end{proof}

%%%
%%%
%%%

\section{Necessary conditions for a  lamination to appear as the limit leaves of the limit lamination
of a sequence of minimal $\theta$-graphs.}
\label{necessaryconditions}

As in Theorem~\ref{laminationproperties}, 
let $D_n$ be a sequence of oriented spanning  minimal $\theta$-graphs
in $\BB$ that converge smoothly in $\BB\setminus Z$ to a
 lamination $\LL$ of $\BB\setminus Z$.
Let $\Rr_n$ be the oriented foliation of $\BB\setminus Z$ consisting of $D_n$ and its rotated
images.   Let $\nu_n$ be the unit normal vectorfied to $\mathcal{D}_n$ compatible with the
orientation.   
Note that $\Rr_n$ will converge to an oriented rotationally invariant foliation $\Rr$
of $\BB\setminus Z$.  In particular, the vectorfields $\nu_n$ converge uniformly on compact
subsets of $\BB\setminus Z$ to the unit normal vectorfield $\nu$ to $\Rr$ compatible with the orientation
of $\Rr$.

The rotationally invariant leaves of $\Rr$ are precisely the rotationally invariant leaves of $\LL$,
that is the leaves of $\LL'$.

In this section, we will prove some additional properties of the collection of rotationally
invariant leaves $\LL'$.   

\begin{proposition}
Suppose that $\Rr$ is an oriented, minimal foliation of an open subset of a Riemannian manifold.
 Then
\[
  \Div \nu=0
\]
where $\nu$ is the unit normal vectorfield to $\Rr$ given by the orientation.
\end{proposition}

\begin{proof}
Let    $\overline\nabla$ denote covariant differentiation with respect to the metric, 
and  let  $\ddiv$ denote the divergence operator on a fixed leaf  of $\Rr$.  We have 
 for any vectorfield $X$:
 \[
 \Div X =\ddiv X +\overline\nabla _\nu X\cdot \nu,
 \]
  so
 \[
 \Div \,\nu =\ddiv\, \nu  + \overline\nabla _\nu \nu\cdot \nu = H+ \frac{1}{2} \nu(\nu\cdot\nu),
 \]
 where $H$ is the mean curvature of the  fixed leaf of $\mathcal R$.  Since all leaves of $\mathcal R$ are minimal and since $\nu$ has unit length, 
 $
 \Div \,\nu =0
 $.
 \end{proof}

\begin{corollary}\label{no-closed-surfaces-corollary}
Suppose that $\Rr$ is an oriented foliation of $\BB\setminus Z$
by surfaces that are minimal with respect to a smooth Riemannian metric on $\BB$, where
$\BB$ is the open unit ball in $\RR^3$.

If $\Sigma$ is a closed, connected, embedded surface in $\BB$,
then $\Sigma\setminus Z$ cannot be a leaf of $\Rr$.
\end{corollary}

\begin{proof}
Let $U$ be the region in $\BB$ bounded by $\Sigma$.
Now $\nu$ is not defined on $Z$, but $Z$ is a closed set with $2$-dimensional Hausdorff measure $0$,
so even if $U\cap Z$ is not empty, 
we can apply the Divergence Theorem~\ref{generalized-divergence-theorem} on $U$ to get:
\[
   \int_{\Sigma} \nu\cdot n\,dA = \int_U\Div \nu = 0, \tag{*}
\]
where $n$ is the unit normal to $\Sigma$ that points out of $U$.
If $\Sigma$ were a leaf of $\LL$, then either $n=\nu$ on $\Sigma$
or $n=-\nu$ on $\Sigma$, so that the left side of~\thetag{*} would be equal 
 to plus or minus the area of $\Sigma$, and thus the area of $\Sigma$ would be $0$, which
 is impossible.
\end{proof}

\begin{corollary}\label{annoying-corollary}
Suppose that the Riemannian metric in Corollary~\ref{no-closed-surfaces-corollary} 
extends smoothly to $\overline{\BB}$.
Let $\Sigma$ be an annulus in $\BB$ such that the two boundary components
of $\Sigma$ are the same smooth, simple closed curve in $\partial \BB$.
Then $\Sigma$ cannot be a leaf of $\LL$.
\end{corollary}

The proof is almost identical to the proof of Corollary~\ref{no-closed-surfaces-corollary}.

\begin{theorem}\label{minimizing-theorem}
Let $\Rr$ be an oriented, minimal foliation of $\BB\setminus Z$ that is rotationally
invariant about $Z$ (with respect to a smooth, rotationally invariant metric on $\BB$), 
and let $\nu$ be the associated unit normal vectorfield compatible with
the orientation.  Let $\Ll'$ be the sublamination consisting of 
 the rotationally invariant leaves of $\Rr$.
Let $U$ be a regular open subset of $\BB$ such that $M:=(\partial U)\cap \BB$ consists
of   leaves of $\Ll'$ on which the normal $\nu$ 
 points out of $U$.
Then $M$ is area minimizing.

Furthermore, if the metric extends smoothly to $\overline{\BB}$ and if $M'$ is another 
area-minimizing surface with 
$\partial M=\partial M'$ (as oriented surfaces in $\overline{\BB}$), then $M'$ is also
made up of oriented leaves of $\Ll'$.
\end{theorem}

Recall that a {\em regular} open set is an open set $U$ 
such that $U={\rm interior}(\overline{U})$.

\begin{proof} 
Case 1: Assume that the metric extends smoothly to $\overline \BB$, 
and that $\overline{M}$ is 
a smooth, embedded manifold with boundary in $\overline{\BB}$.

Let $M'$ be a smoothly embedded, oriented surface in $\BB$ with $\partial M'=\partial M$.
By elementary topology, $M'=(\partial U')\cap \BB$ for some regular open set $U'$ of $\BB$
with 
\begin{equation} \label{Boundaries}
\overline{U'}\cap\partial \BB = \overline{U}\cap\partial \BB. 
\end{equation}
Let 
\[
 \Sigma = \overline{U'}\cap\partial \BB = \overline{U}\cap\partial \BB. 
\]

Let $n$ and $n'$ be the outward-pointing unit normal vectorfields
on $\partial U=M\cup\Sigma$ and on $\partial U'=M'\cup \Sigma$, respectively.
Note that $n=n'$ on $\Sigma$.
Note also that  $n|M=\nu|M$ is the unit normal vectorfield compatible with the
orientation of $M$, and that $n'|M'$ is compatible with orientation of $M'$.

Now $\nu$ is not defined on $Z$. However, 
$Z$ is a closed set with $2$-dimensional Hausdorff measure $0$, so we
can apply the divergence theorem (see Theorem~\ref{generalized-divergence-theorem})
 to $\nu$ on $U$
 to get
\begin{equation}\label{eq:M-inequality}
\begin{aligned}
0
&=
\int_{\tilde U}\Div\nu\,dV
\\
&= 
\int_M \nu\cdot n\,dA + \int_\Sigma \nu\cdot n\,dA.
\end{aligned}
\end{equation} 
Likewise, applying the divergence theorem to $\nu$ on $\tilde U'$ gives
\begin{equation}\label{eq:M'-inequality}
0 = \int_{M'}\nu\cdot n'\,dA + \int_{\Sigma} \nu\cdot n'\,dA.
\end{equation}
Since $n=n'$ on $\Sigma$, 
combining~\eqref{eq:M-inequality} and~\eqref{eq:M'-inequality} gives:
\[
 \int_{M}\nu\cdot n\,dA = \int_{M'}\nu\cdot n'\,dA.
\]
The left side equals the area of $M$ since, by hypothesis, $\nu\equiv n$ on $M$.
Thus
\begin{equation}\label{eq:inequality}
\area(M) = \int_{M'}\nu\cdot n'\,dA \le \area(M'),
\end{equation}
with equality if and only if $\nu\equiv n'$, i.e., if and only if $M'$ is also a leaf of $\LL'$.
This proves that $M$ is area-minimizing, and it also proves the last assertion
(``furthermore\dots") of the theorem.

{\bf Case 2}: The general case.
Let $\BB_1\subset \BB_2\subset \dots$ be an exhaustion of $\BB$ by 
open balls centered at the origin such that 
for each $i$, $\partial \BB_i$ is transverse to $M$.
Then (by Case 1), $M\cap \BB_i$ is area minimizing in $\BB_i$ (i.e., it 
has area less than or equal to the area of any other surface in $\BB_i$ with the same boundary.)
Thus (by definition), $M$ is area minimizing in $\BB$.
\end{proof}

We give two simple applications of Theorem~\ref{minimizing-theorem}
that will be used in the next section.  For these corollaries, we assume that
the metric extends smoothly to $\overline{\BB}$.

\begin{corollary}\label{minimizing-disks-corollary}
Suppose $\Ll'$ contains two disks $D$ and $D'$ such that 
 either $\nu$ points out of the  region $\Omega$ between $D$ and $D'$ on $D\cup D'$,  
or it points into that region on $D\cup D'$.
Then the two disks are area minimizing as an integral current.
If there is an area-minimizing annulus with the same boundary as the disks, then
it must also be a leaf of $\Ll'$.
\end{corollary}

\begin{proof}
If $\nu$ on $D\cup D'$ points out of $\Omega$,  let $U=\Omega$.
If it points into $\Omega$, let $U = \BB\setminus \overline{\Omega}$.
Now apply Theorem~\ref{minimizing-theorem}.
\end{proof}

\begin{corollary}\label{minimizing-annulus-corollary}
Suppose that $\Ll'$ contains an annulus $M$.  Then $M$ is area
minimizing as an integral current.  If $\partial M$ bounds another
area-minimizing surface, then it must also be a leaf or union of leaves of $\Ll'$. 
\end{corollary}

\begin{proof}
Note that $M$ divides $\BB$ into two components: we let $U$ be the component
that such that $\nu$ on $M$ points out of $U$.  Now apply Theorem~\ref{minimizing-theorem}.
\end{proof}

\begin{remark}
We conjecture the following partial converse to Theorem~\ref{minimizing-theorem}.
 Suppose that one has a  finite collection  of area-minimizing, rotationally invariant minimal surfaces. Let $\MM$ be the augmentation of this collection to include all
 area-minimizing, rotationally invariant minimal surfaces with the same boundary. Then
$\MM$ can be realized as the rotationally invariant leaves $\Ll '$ of a lamination $\Ll$ that is a limit lamination of a sequence of spanning minimal $\theta$-graphs in $\BB\setminus Z$. 
\end{remark}

\section{Specifying the rotationally invariant leaves \\ of a limit lamination}\label{SpecifyingLeaves}

In this section, we work with the open unit ball $\BB$ in $\RR^3$ and with a smooth Riemannian 
metric $g$ on $\overline{\BB}$ such that
\begin{itemize}
\item the mean curvature of $\partial \BB$ is nonzero and points into $\BB$.
\item The metric is rotationally invariant about $Z$, and also invariant
under $\mu(x,y,z)=(x,y,-z)$.
\end{itemize}
\begin{definition}For $0<a<1$,  let 
\[
     c(a)= c^+(a)\cup c^-(a),
\]
where
$c^\pm(a)$ are the circles $(\partial \BB) \cap \{z= \pm a\}$. 
We orient  $c^-(a)$  by $d\theta$ and $c^+(a)$ by $-d\theta$.
Let 
\[
     \MM(a)
\]
 be the set of rotationally invariant area-minimizing surfaces bounded by $c(a)$.
\end{definition}

The hypotheses imply that $\MM(a)$ is nonempty for every $a\in (0,1)$.
Each surface in $\MM(a)$ (indeed, any rotationally invariant surface bounded by $c(a)$)
is either a pair of disks or an annulus.  

If $M\in \MM(a)$, then the area of $M$ is less than the area of the
annular component of $\partial \BB\setminus c(a)$. 
It follows that if $a\in (0,1)$ is close to $0$, then $M$ is an annulus.
(By the strict mean convexity of $\overline{\BB}$, $M\setminus\partial M\subset \BB$.)

Likewise,  the area of a surface $M\in\MM(a)$ is less than the area of the union
of the two simply connected components of $\partial \BB\setminus c(a)$.
In particular, if $a\in (0,1)$ is close to $1$, then the area of $M$ is nearly $0$.
It follows that if $a$ is close to $1$, then $M$ is a pair of disks.
(For if $a$ is close to $1$, then any minimal annulus bounded by $c(a)$ would contain
points from far $c(a)$, and thus by monotonicity would have area bounded away
from $0$.)

By a standard cut-and-paste argument, the surfaces in $\MM(a)$ are disjoint
 from each other, except at their common boundary. By similar reasoning,  if $a\ne a'$, the surfaces
in $\MM(a)$ are disjoint from the surfaces in $\MM(a')$.   Thus the collection of surfaces $\MM(a)$, $0<a<1$,
forms a minimal lamination of $\BB$.   Figure~\ref{AreaMinimizingCatenoidsAndDisks}  shows that lamination for the Euclidean metric.

Note that if $0<a<b<1$ and if $\MM(b)$ contains an annulus, then $\MM(a)$ contains {\em only}
annuli.  For otherwise $\MM(a)$ would contain a pair of disks, and those disks would intersect
the annulus in $\MM(b)$, which is impossible.

 Consequently, there is an $a_{\rm crit} \in (0,1)$ such that
\begin{enumerate}
\item\label{disks-item} if $a_{\rm crit}<a<1$, then $\MM(a)$ contains at least one pair
of disks, but no annuli.
\item\label{annuli-item} if $0<a<a_{\rm crit}$, then $\MM(a)$ contains at least one annulus,
but no pairs of disks. 
\item\label{both-item} $\MM(a_{\rm crit})$ contains at least one pair of disks, and it contains at least
one annulus.
\end{enumerate}
(Note that~\eqref{both-item} follows from~\eqref{disks-item} and~\eqref{annuli-item},
since the limit of area-minimizing annuli is also an area-minimizing annulus, and similarly
for pairs of disks.)

For the Euclidean metric, for each $a\le a_{\rm crit}$, $\MM(a)$ contains exactly one minimal annulus,
and for each $a\ge a_{\rm crit}$, $\MM(a)$ contains exactly one pair of minimal disks.
But for general metrics, a given $\MM(a)$ might contains multiple minimal annuli and/or multiple pairs of disks.

\begin{definition} If $T$ is a relatively closed subset of $(0,1)$, 
let $c(T)$ be the collection of circles in $\partial \BB$ given by
\[ 
c(T) =\bigcup_{a\in T}  c(a),
\]
  and let $\MM(T)$ be the lamination of $\BB$ given by
\[
 \MM(T)= \bigcup_{a\in T} \MM(a).
\]
We let 
\[
\MM=\MM(0,1).
\]
\end{definition}

\begin{theorem}\label{realizing-M(T)}
Consider a smooth Riemannian metric on $\overline{\BB}$ such that
\begin{enumerate}
\item the mean curvature of $\partial {B}$ is nonzero and points into $\BB$.
\item the metric is invariant under $(x,y,z)\mapsto (x,y,-z)$ and under rotations about $Z$.
\end{enumerate}
Let $T$ be a relatively closed subset of $(0,1)$.
Then there exists a sequence of spanning minimal $\theta$-graphs in $\BB\setminus Z$
that converge to a limit lamination $\LL$ whose rotationally invariant 
leaves are given by $\MM(T)$.
\end{theorem}

 \begin{remark} 
 More precisely, the rotationally invariant 
 leaves of $\LL$ are 
 the annuli in $\MM(T)$ together with the disks in $\MM(T)$ with their centers removed. 
\end{remark}

\begin{proof}
\newcommand{\Gg}{\mathcal{G}}
First suppose that $0\notin \overline{T}$.
Consider the collection $\Gg$ of $\theta$-graphs $\gamma$ in $\partial \BB$
with the following properties:
\begin{itemize}
\item $\gamma$ is invariant under the reflection $\mu(x,y,z)=(x,y,-z)$.
\item $\frac{d\theta}{dz}$ is positive on $\gamma\cap\{z<0\}$
(and therefore negative on $\gamma\cap\{z>0\}$).
\end{itemize}
Then there is a sequence of curves
$\gamma_i$, $i=1,2,\dots$, in $\Gg$ converging
smoothly to a lamination $\Cc$ of $\partial \BB\setminus Z$ such that the rotationally invariant leaves
are precisely the circles in $c(T)$.

By Theorem~\ref{UniqueEmbeddedDisk} (applicable because we are assuming that $\overline B$ is mean convex),  for each $\gamma_i$ there exists
a   unique, smooth, embedded minimal $\theta$-graph $D_i$ with boundary $\gamma_i\cup  (Z\cap\overline \BB)$. Because this boundary is $\mu$-invariant, uniqueness implies that $D_i$ is also $\mu$-invariant.
By passing to a subsequence, we can assume that the $D_i$
converge smoothly to a lamination $\Ll$ of $\BB\setminus Z$. Of course $\Ll$ must be $\mu$-invariant.

To prove the theorem, we must prove that every rotationally invariant leaf of $\Ll$
is in $\MM(T)$, and, conversely, that each surface in $\MM(T)$ is a leaf of $\Ll$.
\newline\newline
\noindent
{\bf Step 1: Proof that every rotationally invariant leaf of $\Ll$ is in $\MM(T)$.}
Suppose that $L$ is a rotationally invariant leaf of $\Ll$.
Then $L$ must be a punctured disk or an annulus.

Case 1: $L$ is a punctured disk.   By Theorem~\ref{smoothboundarybehavior},
 the boundary circle 
of $L$ must be a leaf of $\Cc$, so it must be one of the two circles in $c(a)$ for
some $a\in T$.   By symmetry, $\mu(L)$ is also a leaf of $\Ll$.  
The two boundary circles of $L\cup \mu(L)$ are $c(a)$.   
 By Corollary~\ref{minimizing-disks-corollary}, $L\cup\mu(L)$ is area
minimizing.  Thus $L\cup \mu(L)\in \MM(a)$.

Case 2: $L$ is an annulus.   
By Theorem~\ref{smoothboundarybehavior}, the two boundary circles of $L$ must both be
circles in the family $c(T)$.   
Note the circles must be oppositely oriented.
Therefore one boundary circle is $c^+(a)$ for some $a\in T$, and the other is $c^-(b)$
for some $b\in T$.

We claim that $a=b$.  For otherwise, $L$ and $\mu(L)$ would be two leaves of $\Ll$
that intersect along a circle at height $0$, which is impossible.  
Thus $\partial L=c(a)$ for some $a\in T$.   
By Corollary~\ref{minimizing-annulus-corollary}, $L$ is area-minimzing.
Therefore $L\in \MM(a)$.

This completes the proof that each rotationally invariant leaf in $\Ll$ is in $\MM(a)$ for
some $a\in T$.
\newline
\newline
{\bf Step 2: Proof that every surface in $\MM(T)$ is a rotationally invariant leaf in $\Ll$.}
Suppose that $a\in T$.
By Corollary~\ref{smoothboundarycorollary}, 
 there is some rotationally invariant leaf $L$ of $\Ll$ such that $c^+(a)$ is a
component of $\partial L$.
If $L$ is an annulus, then (as we have proved above), 
$\partial L=c(a)$; in this case, let $\Sigma=L$.
If $L$ is a disk, then $\mu(L)$ is also in $\Ll$ (by $\mu$-symmetry);
In this case, we let $\Sigma=L\cup\mu(L)$.

We have shown: if $a\in T$, then $c(a)$ bounds a rotationally invariant surface $\Sigma$ consisting
of one leaf (an annulus) or two leaves (both disks) in $\Ll$.
By Corollaries~\ref{minimizing-disks-corollary} and~\ref{minimizing-annulus-corollary},
 $\Sigma$ is area minimizing, so $\Sigma\in \MM(a)$.
If $\MM(a)$ contains another surface $\Sigma'$, then $\Sigma$ together with $\Sigma'$
bound a region $\Omega$.  By Theorem~\ref{minimizing-theorem}, since $\Sigma$ is a rotationally invariant leaf (or
pair of leaves) in $\Ll$, $\Sigma'$ must also be in $\Ll$. 
Thus every surface $\Sigma'$ in $\MM(a)$ belongs to $\Ll$.
This completes the proof assuming that $0\notin\overline{T}$.

Now suppose that $0\in \overline{T}$.
In this case, no sequence $\gamma_i\in \GG$ can converge smoothly to a lamination
that includes the circles $c(T)$.
For if  $\gamma_i$ in  $\GG$  converges smoothly to a lamination
$\Cc$ of $\partial \BB\setminus\{p^+, p^-\}$, then $\Cc$ contains a leaf that crosses
the equator perpendicularly, which implies that $\Cc$ cannot contain circles arbitrarily near
the equator.  

However, even if $0\in\overline{T}$, we can find a sequence of curves $\gamma_i$ in $\GG$
that converge smoothly in 
\[
\partial \BB \setminus (\{p^+,p^-\}\cup \{z=0\}) \tag{*}
\]
to a lamination of \thetag{*} whose rotationally invariant leaves are precisely the circles in $c(T)$.
The rest of the proof is almost exactly the same as the proof when $0\notin \overline{T}$.
\end{proof}

 \begin{remark}  
Let $W$ be the open cylinder $\{(x,y,z): x^2+y^2<1\}$ with the Euclidean metric.
In~\cite{hoffman-white-sequences}, the authors prove that given
any closed subset $T$ of $Z$, there is a sequence of spanning
$\theta$-graphs in $W\setminus Z$ that converge
to a limit lamination whose rotationally invariant leaves
are precisely the disks $W\cap\{z=c\}$, $c\in T$.
\end{remark}

\section{The hyperbolic case I.  Existence of $\theta$-graphs with  prescribed boundary at infinity: Theorem~\ref{UniqueEmbeddedDisk}  in the hyperbolic case}\label{hyperbolic1}

We will  extend  the existence result, Theorem~\ref{UniqueEmbeddedDisk} of
 Section~\ref{ExistenceUniqueness}  (and Theorem~2 of \cite{hoffman-white-sequences}), to hyperbolic space $\HH ^3$.    In this section,  $\BB$ will denote the open
unit ball centered at the origin in $\RR^3$.  
We will be interested in surfaces in $\BB$ that are {\bf hyperbolically minimal}, i.e. minimal with respect to the hyperbolic (Poincar\'e) metric
\[
        ds^2 =\frac{4(dx_1^2+dx_2^2+dx_3^2)}{(1-|x|^2))^2}
\]
on $\BB$.  This metric is clearly rotationally symmetric around  any axis of the ball, 
in particular the $x_3$-axis $Z$.   Note that Theorem~\ref{UniqueEmbeddedDisk}
 does not directly apply here because the metric does not extend smoothly 
 to  $\overline{\BB}$,  and the boundary 
 (the unit sphere---at infinite distance from any point of $\BB$) 
 is not mean convex in the ordinary sense.
 
 In what follows, for any subset $S$  of $\overline{\BB}$,  the sets  $\overline{S}$ and $\partial S$
will continue to denote the closure of $S$ and the boundary of $S$ in $\overline{\BB}$ with respect the Euclidean metric.
We will refer to     $\partial S \cap \partial \BB$
as the {\bf ideal boundary} of $S$.
 We will write  $I=Z\cap \BB$ and observe that the ideal boundary of $I$ is 
 equal to $Z\cap \partial \BB=\{(0,0,\pm1)\}$.
 
\begin{theorem} \label{hyperbolic existence}
Let $\gamma$ be a smooth curve in $\partial \BB$ joining $p^-$ to $p^+$ such 
that $\gamma$ intersects each rotationally invariant curve in $\partial \BB$ exactly once,
and such that the curve $C:=\gamma\cup \rho_Z\gamma$ is smooth.
Let $\Gamma$ be the union of $\gamma$ with $Z$.
 Then  $\Gamma$ bounds a spanning hyperbolically minimal $\theta$-graph $D$ such that
 $\overline{D\cup \rho_ZD}$ is a smoothly embedded disk with boundary $C$.
 \end{theorem}

\begin{proof}
Let $\Bb_n$ be a sequence of nested open balls centered at the origin such that $\overline{\Bb_n}\subset \BB$ and such
that $\bigcup_n\Bb_n=\BB$.
Let $\Gamma_n$ be the image of $\Gamma$ under the Euclidean homothety that takes $\BB$ to $\Bb_n$.

By Theorem~\ref{UniqueEmbeddedDisk}, the curve $\Gamma_n$ bounds a unique
spanning $\theta$-graph $D_n$ that is minimal with respect to the Poincar\'e metric.
Its rotated images about $Z$ foliate $\Bb_n\setminus Z$.

By Theorem~\ref{laminationproperties},  we can assume (by passing to a subsequence) that the $D_n$ converge smoothly on compact subsets of $\BB\setminus Z$
to a minimal lamination $\Ll$ of $\BB\setminus Z$. (Theorem~\ref{laminationproperties} assumes that all the $D_n$
lie in the same domain $\BB\setminus Z$.
 Here we have expanding domains $\Bb_n\setminus Z$. The proof is the same, requiring only the choice of subsequences at each stage.)

Let $\partial \Ll = \overline{\Ll}\setminus \Ll$.
Note that $\partial \Ll$ is a subset of $(\partial \BB)\cup I$. 

\vspace{0.15in}

\noindent {\bf Claim 1.} {\em   $\partial \Ll$ is contained in $\Gamma$.}

\begin{proof} Every point in $p\in(\partial \BB)\setminus \Gamma$ is contained in an open Euclidean ball $U$
that is disjoint from $\Gamma$ and that meets $\partial \BB$ orthogonally.  Note that $U$ is disjoint from each 
$\Gamma_n$.  Note also that $U\cap \BB$ can be foliated by
nested totally geodesic surfaces (the boundaries of smaller balls) that meet $\partial \BB$ orthogonally and converge in the Euclidean metric to $p$. 
By the the maximum principle,  $U$ is disjoint from $D_n$.    This proves the claim.
\end{proof}

Because $\Gamma$ contains no circles it follows from Claim~1 that  $\Ll$ contains no leaves that are rotationally invariant. We now use
the properties of $\Ll$ that were proved in Theorem~\ref{laminationproperties}:
 \begin{itemize}
 \item By Property~\eqref{main-theorem-item3}, $\Ll$ contains no limit leaves;
 \item By Property~\eqref{main-theorem-blowup-item}, 
 there is no curvature blowup in  $\BB\setminus I$ and, by 
 Property~\eqref{main-theorem-puncture-item}, there is no curvature blowup on  $I$; 
 \item Consequently, by Property~\eqref{main-theorem-schwarz-item}, 
 there is  a single leaf  $D$ of $\Ll$ that contains $I$ in its closure,
  and that leaf is a spanning $\theta$-graph.
 \end{itemize}
 It follows from Property~\eqref{main-theorem-item4} that  $D$ is the only leaf of $\Ll$.
 
Because there is no curvature blowup on $I$, the local boundedness of the curvatures of the $D_n$ implies that $D\cup I$ is a smooth manifold with boundary.

\vspace{0.15in}

\noindent {\bf Claim 2.} {\em  Let $B(D,p)$ denote the norm of the second fundamental form of $D$ with respect to the hyperbolic metric.
Then $B(D,\cdot)$ is bounded above on $D$.}

\begin{proof}[Proof of Claim~2] 
If  Claim~2 is false,  then there is a sequence of points $p_n\in D$ such that $B(D,p_n)\to\infty$.
By passing to a subsequence, we can assume that $p_n$ converges (in the Euclidean sense)
to a point $p\in \overline{\BB}$.   
Since $D\cup I$ is a smooth manifold with boundary, $p\in \partial \BB$.
Since the $R_\theta D$ foliate $\BB\setminus Z$, $D$ is stable, and stability yields the following estimate:
\begin{equation}\label{stability}
    B(D,p_n)\, \min \{ 1, \dist(p_n, \partial D)\} \le c,
\end{equation}
where $c$ is a constant independent of $n$, and $\dist$ denotes distance in the hyperbolic metric. (See \cite{Schoen}.)
Since the ideal boundary $\partial \BB$ of $\BB$ is infinitely far from $p_n$ (in hyperbolic distance),
$$\dist(p_n, \partial D)=\dist(p_n, (\partial D)\cap \BB)=\dist(p_n, I),$$
so \eqref{stability} becomes
\[
     B(D,p_n)\, \min \{ 1, \dist(p_n, I)\} \le c.
\]
Since  we are assuming that $B(D, p_n)\to\infty$, this impies that $\dist(p_n, I)\to 0$. 
Hence $p$ is one of the points of $I$. However, we have established that $D\cup I$ is a smooth manifold with boundary, so $p$ must lie in $\partial I$. That is, $p=(0,0 ,1)$ or $p=(0,0,-1)$.  Passing to a subsequence, we may assume without loss of generality that  that $p=(0,0,1)$,  the North Pole, and $p_n\rightarrow p$.

Let $f_n:\overline{\BB}\to\overline{\BB}$ be a Mobius transformation (i.e., a hyperbolic isometry) with the property that $p_n' := f_n(p_n)$ lies on the plane $z=0$ and that $f_n(Z)=Z$.  Let $D_n$ and $\gamma_n$ be the images
of $D$ and $\gamma$, respectively,  under $f_n$. 
  Since 
\[
        \dist(p_n', Z)=\dist(f_n(p_n), f_n(Z))=\dist(p_n,Z),
\]
and $p'_n$ lies on the plane $\{z=0\}$
the $p_n'$ converge to 
the origin $O=(0,0,0)$.
Note that $\gamma_n$ converges smoothly 
(except at the South Pole) to a great semicircle $S$
joining the North and South Poles.  

 By passing to a subsequence, we may assume that the $D_n$ converge
to a minimal lamination $\Ll'$ of $\BB\setminus Z$.  As before, the ideal boundary of the lamination is contained
in $S$ and therefore does not contain any horizontal circles.   Thus $\Ll'$ does not contain any rotationally symmetric
leaves. That is, there are no limit leaves. Thus the curvatures of the $D_n$ are uniformly bounded
on compact subsets of $\BB$, contradicting the fact that $B(D_n, p_n')\to \infty$ and that $p_n' \to O$.
This completes the proof the claim.
\end{proof}

We now complete the proof of Theorem~\ref{hyperbolic existence}. Let $$\mathcal{D}= D \cup \rho_ZD \cup I.$$  
Then $\mathcal{D}$ is an embedded minimal disk whose ideal boundary
is the smooth, simple closed curve  $\gamma \cup \rho_Z (\gamma)$,
 and whose principal curvatures are uniformly bounded.
 It follows from the  work of Hardt and Lin \cite{HardtLin}  that $\overline{\mathcal{D}}$ is a $C^1$-manifold with boundary and must meet the ideal boundary orthogonally. Based on this work,  Tonegawa \cite{tonegawa} was able to prove that in fact $\overline{\mathcal{D}}$ is a smooth manifold with boundary.
(This assertion requires some explanation. First, Hardt and Lin  assume  that $\mathcal{D}$ is a hyperbolic-area-minimizing rectifiable current.  Their proof works equally well if instead one
assumes that $\mathcal{D}$ is a smooth minimal surface whose principal curvatures  are bounded, and we have established these bounds in Claim~2.  Such boundedness
easily implies Lemma~2.1 of \cite{HardtLin} , which  establishes the essential property of surfaces  necessary for their proof of their result. Second, the main theorem of \cite{HardtLin} states that near the boundary, $\overline{\mathcal{D}}$ is 
a union of sheets,  each of which is a smooth manifold with boundary.  But in our case there is clearly only one sheet since
$\mathcal{D}$ intersects each horizontal circle centered on $Z$ exactly once.)
\end{proof}

\begin{proposition} 
\label{hyperbolic-uniqueness}
Let $D$ be {\color{black} a} spanning hyperbolically minimal $\theta$-graph as in  
Theorem~\ref{hyperbolic existence}.
Let $M$ be a  hyperbolically minimal surface embedded in  $\BB\setminus Z$ such that
$\partial M=\partial D$ and such that 
$\overline{M\cup\rho_Z M}$ is a $C^1$ manifold with boundary.
Then $M=D$.
\end{proposition}

The proof of Proposition~\ref{hyperbolic-uniqueness} is exactly
the same as the proof of the uniqueness assertion in Theorem~\ref{UniqueEmbeddedDisk}.

%%%
%%%
%%%

\section{The Hyperbolic case II. Necessary Conditions for a lamination to appear as  limit leaves of a limit lamination: Section~\ref{necessaryconditions} in  the  hyperbolic case.}\label{necessaryconditionshyperbolic}
\stepcounter{theorem}

\newcommand{\Cyl}{\operatorname{Cyl}}

The statements and proofs in Section~\ref{necessaryconditions} involved comparing areas of rotationally invariant
 surfaces in $\BB$ with boundaries in $\partial \BB$. 
If we endow $\BB$ with the Poincar\'e metric, then the areas of such surfaces are
infinite, so comparing them becomes problematic.
  We get around this problem by working with suitable compact exhaustions
of the surfaces.  
Let $\Cyl(s)$ denote the points in $\BB$ that are at (hyperbolic) distance at most $s$
from $Z\cap \BB$. Inspired by \cite{collin-rosenberg-harmonic}, where horocycles are used to cut off ends of divergent geodesics in order to define a Jenkins-Serrin-like condition for  minimal graphs in $H^2\times \RR$ with infinite boundary values, we will make regions and surfaces finite by clipping them with  the cylinders  $\Cyl(s)$.

We will use the following fact about catenoids in hyperbolic space.

\begin{theorem}\label{ribbon-theorem}
Let $C$ be a half-catenoid with axis $Z$, and let $D$ be another half-catenoid
or a totally geodesic disk such that $C$ and $D$ have the same ideal boundary circle.
For $s$ large, let $\Sigma(s)$ be the portion of $\partial \Cyl(s)$ between $C$ and $D$.
Then 
\[
  \lim_{s\to\infty} \area(\Sigma(s)) = 0.
\]
\end{theorem}

(A half catenoid is, by definition, one of the two components obtained from a catenoid by removing
its waist, i.e., its unique closed geodesic.)
See  Appendix~\ref{hyperbolic-ribbons}, specifically Corollary~\ref{length-I-theta} and Remark~\ref{ribbon-theorem-proof}  for  a proof of Theorem~\ref{ribbon-theorem}.

\begin{theorem}\label{hyberbolic-minimizing-theorem}
 Consider the open unit ball $\BB\subset\RR^3$ with the Poincar\'e metric.
Let $\Rr$ be an oriented, minimal foliation of $\BB\setminus Z$ that is rotationally
invariant about $Z$, and let $\nu(\cdot)$ be the associated unit normal vectorfield compatible with
the orientation.
  Let $\Ll'$ be the sublamination consisting of rotationally invariant leaves of $\Rr$.

Let $U$ be a regular open subset of $\BB$ such that $M:=(\partial U)\cap \BB$ consists
of leaves of $\Ll'$ on which the normal $\nu$ 
 points out of $U$.
Then $M$ is area-minimizing.

Furthermore, if $M$ consists of finitely many leaves, and if $M'$ is another rotationally invariant, area-minimizing surface with 
$\partial M=\partial M'$ (as oriented surfaces in $\overline{\BB}$), then $M'$ is also
made up of oriented leaves of $\Ll'$.
\end{theorem}

Of course $M$ has infinite area.  Recall that such a surface is said to be area-minimizing
provided every compact portion of it is area-minimizing.

\begin{proof}
Let $\BB_r=\BB(0,r)$ be the ball of Euclidean radius $r$ centered at $0$.
Thus the hyperbolic radius of $\BB_r$ tends to $\infty$ as $r\to 1$.
Note that
\[
    \Rr_r:= \{L\cap \BB_r: L\in \Rr\}
\]
is a rotationally invariant foliation of $\BB_r\setminus Z$.  
Applying Theorem~\ref{minimizing-theorem} to $\Rr_r$, $M\cap \BB_r$, and $U\cap \BB_r$, we see
that $M\cap \BB_r$ is area minimizing.    Since this is true for each $r<1$, the 
surface $M$ is area minimizing.

To prove the ``furthermore" assertion, 
let $M'$ be a rotationally invariant area-minimizing surface with $\partial M'=\partial M$.
By elementary topology, there is regular open set $U'$ such that $M'=(\partial U')\cap \BB$
and such that $\overline{U'}\cap \partial \BB = \overline{U}\cap \partial\BB$.

Note that $M'\cap \Cyl(s)$ has the same boundary as the surface consisting
of $M\cap \Cyl(s)$, $(U\setminus U')\cap \partial \Cyl(s)$, and $(U'\setminus U)\cap \partial \Cyl(s)$
(provided the latter two surfaces are oriented suitably).   Thus, since $M'$ is area-minimizing,
\begin{equation}\label{eq:M'-minimizes}
\area(M'\cap\Cyl(s))
\le
\area(M\cap \Cyl(s)) 
+ \area((U\Delta U')\cap \partial \Cyl(s)),
\end{equation}
where $U\Delta U'=(U\setminus U')\cup (U'\setminus U)$ denotes the symmetric difference of $U$ and $U'$.
By Theorem~\ref{ribbon-theorem},
\begin{equation}\label{dwindle}
\area((U\Delta U')\cap \partial \Cyl(s)) \to 0
\end{equation}
as $s\to\infty$, so by~\eqref{eq:M'-minimizes},
\begin{equation}\label{eq:M'-minimizes-o(1)}
\area(M'\cap\Cyl(s))
\le
\area(M\cap \Cyl(s)) + o(1),
\end{equation}
where $o(1)$ denotes any quantity that tends to $0$ as $s\to\infty$.

Note that $M$ and $M'$, and therefore also $U\setminus U'$ and $U'\setminus U$, lie within
a bounded hyperbolic distance $d$ of $\BB\cap\{z=0\}$.   Fix an $h>d$.
Let $\Cyl(s,h)$ denote the set of points in $\Cyl(s)$ that are at hyperbolic distance
less than $h$ from
 $\BB\cap\{z=0\}$.
Let  ${U(s,h)=U\cap{ \Cyl(s,h)}}$.
Now apply the divergence theorem to $\nu$ on ${U(s,h)}$:
\begin{equation}\label{eq:first-divergence}
\begin{aligned}
0 &= \int_{{U(s,h)}}\Div \nu \, dv
\\
&= \int_{M\cap {\Cyl(s,h)} }\nu\cdot n\,dA + \int_{U \cap \partial{ \Cyl(s,h)}} \nu\cdot n_{{ \Cyl(s,h)}}\,dA.
\end{aligned}
\end{equation}
Similarly, applying the divergence theorem to $\nu$ on $U'(s,h)=U'\cap {\Cyl(s,h)}$ gives
\begin{equation}\label{eq:second-divergence}
0 
= \int_{M'\cap { \Cyl(s,h)}} \nu\cdot n'\,dA + \int_{U' \cap \partial { \Cyl(s,h)}} \nu\cdot n_{\Cyl(s,h)}\,dA.
\end{equation}
Combining~\eqref{eq:first-divergence} and~\eqref{eq:second-divergence} gives
\begin{align*}
&\int_{M\cap { \Cyl(s,h)}} \nu\cdot n\,dA 
+ \int_{(U\setminus U') \cap \partial { \Cyl(s,h)}} \nu\cdot n_{{ \Cyl(s,h)}}\,dA
\\
&\qquad=
\int_{M'\cap \Cyl(s,h)} \nu\cdot n'\,dA 
+ \int_{(U'\setminus U) \cap \partial \Cyl(s,h)} \nu\cdot n_{\Cyl(s,h)}\,dA.
\end{align*}
By choice of $h$, none of these terms is changed if we replace $\Cyl(s,h)$ by $\Cyl(s)$:
\begin{align*}
&\int_{M\cap \Cyl(s)} \nu\cdot n\,dA 
+ \int_{(U\setminus U') \cap \partial \Cyl(s)} \nu\cdot n_{\Cyl(s)}\,dA
\\
&\qquad=
\int_{M'\cap \Cyl(s)} \nu\cdot n'\,dA 
+ \int_{(U'\setminus U) \cap \partial \Cyl(s)} \nu\cdot n_{\Cyl(s)}\,dA.
\end{align*}
Since $\nu\equiv n$ on $M$,  using ~\eqref{dwindle}, we have
\begin{align*}
\area(M\cap \Cyl(s)) 
&= \area(M'\cap \Cyl(s)) + \int_{M'\cap\Cyl(s)} (\nu\cdot n' - 1)\,dA + o(1)
\\
&\le \area(M\cap \Cyl(s)) + \int_{M'\cap\Cyl(s)} (\nu\cdot n' - 1)\,dA + o(1)
\end{align*}
by~\eqref{eq:M'-minimizes-o(1)}. 

Subtracting $\area(M\cap\Cyl(s))$ from both sides and then letting $s\to \infty$ 
gives
\[
  0 \le \int_{M'} (\nu\cdot n' -1 )\,dA
\]
which implies that $\nu\equiv n'$ on $M'$. This implies that $M'$ consists of  rotationally invariant leaves  of $\Rr$.  
\end{proof}

%%%%%%

\section{The Hyperbolic case III. Specifying the rotationally invariant leaves of a limit lamination: Section~\ref{SpecifyingLeaves} in  the  hyperbolic case.}
\label{hyperbolic3}

For a relatively closed subset $T\subset (0,1)$, we defined in Section 6.1 a lamination
 $\mathcal C(T)$ of $\partial \BB\setminus\{p^+,p^-\}$ and a lamination $\MM(T)$ of
  $\BB\setminus Z$.
 
 \begin{theorem}\label{hyperbolic-realizing-M(T)}
Let $ \BB$  be the open unit ball with the Poincar\'e  metric. Let $T$ be a relatively closed
 subset of $(0,1)$.
There exists a sequence of  spanning minimal $\theta$-graphs in $\BB\setminus Z$
that converge to a limit lamination $\LL$ whose rotationally invariant 
leaves are precisely $\MM(T)$.
\end{theorem}

This is the hyperbolic version of Theorem~\ref{realizing-M(T)}. That theorem is for Riemannian
metrics on $\BB$ that extend smoothly to $\overline{\BB}$, 
 something that is not true 
for the Poincar\'e metric. 
Nevertheless, we can use the results of the previous sections to prove this result. 

\begin{proof}
We follow the proof of Theorem~\ref{realizing-M(T)}. Start with a sequence $\gamma_i\subset \partial B$ of $\theta$-graphs with the bulleted properties that define 
$\GG$. Choose them so that they that converge 
 to a lamination  $\mathcal C$  of $\partial B\setminus Z$ whose 
rotationally invariant leaves are precisely the circles  in $c(T)$,
where the convergence is smooth except possibly where $z=0$.
By Theorem~\ref{hyperbolic existence}, we may assert the existence of a smooth, embedded, minimal $\theta$-graph $D_i$  with boundary $\gamma_i\cup (Z\cap\overline{\BB})$. Since this boundary is $\mu$-invariant it follows from Proposition~{7.2} that $D_i$ is also $\mu$-invariant.   Passing to a subsequence, we may assume that these $\theta$-graphs  converge smoothly to a lamination $\LL$ of $\overline B\setminus Z$. This limit lamination is also $\mu$-invariant.  We must show that the limit leaves of $\LL$ are precisely the rotationally invariant surfaces in $\MM(T)$.

The limit leaves of $\LL'\subset \LL$ together with the rotations about $Z$ of the non-limit leaves of $\LL$ form an oriented minimal  foliation $\Rr$ of $\BB\setminus Z$ that is rotationally invariant about $Z$. Therefore, we may use Theorem~\ref{hyberbolic-minimizing-theorem}. This theorem can be easily used to show that 
Corollaries~\ref{minimizing-disks-corollary} and \ref{minimizing-annulus-corollary} hold in hyperbolic space.  The arguments in Steps~1 and ~2 of the proof of 
Theorem~\ref{realizing-M(T)} are  now directly applicable to our situation, using
Theorem~\ref{hyberbolic-minimizing-theorem} where Theorem~\ref{minimizing-theorem}
is invoked.
\end{proof} 

 \begin{remark} To be precise, the limit leaves are the annuli, if any, in $\MM(T)$ together with the disks in $\MM(T)$, if any,  with their centers removed. By not removing  the centers, we may consider $\MM(T)$ as a lamination of $\BB$. 
\end{remark}

As a application of Theorem~\ref{hyperbolic-realizing-M(T)}, let $a$ be small enough
so that $\MM(a)$ consists of one or more catenoids.
 (See~\eqref{annuli-item} of Section~\ref{SpecifyingLeaves}.) 
  Theorem~\ref{hyperbolic-realizing-M(T)} above tells us we may realize  $\MM(a)$ as the limit leaves of a limit lamination of $\BB\setminus Z$.
Doubling the nonlimit leaves of the limit lamination by reflection in $Z$ produces a
lamination of $\BB$ with the same limit leaves $\MM(a)$, one nonlimit leaf in the component of $\BB\setminus\cup \MM(a)$ that contains $\BB\cap Z$, and two congruent leaves in every other component of  $\BB\setminus \cup \MM(a)$.  Consequently:

\begin{theorem}\label{CCCounterexample} There exist complete, embedded, simply connected minimal surfaces
in hyperbolic space that are not properly embedded. In particular for every area-minimizing catenoid $C$ in hyperbolic space, there exist two complete, noncongruent, simply connected, embedded  minimal surfaces (one one either side of $C$) that have $C$ in their closure.
\end{theorem}

%%%%%
%%%

%%%%%

\appendix
\section{The divergence theorem}

\begin{theorem}[Generalized Divergence Theorem]\label{generalized-divergence-theorem}
Suppose that $\Omega$ is a domain with compact closure and with piecewise smooth
boundary in a Riemannian $(m+1)$-manifold.  Suppose that $K$ is a compact subset
of $\overline{\Omega}$ with Hausdorff $m$-dimensional measure $0$, and that
$\nu$ is a bounded $C^1$ vectorfield on $\overline{\Omega}\setminus K$
such that $\int |\Div \nu|\,dV< \infty$.  Then
\[
   \int_\Omega\Div\nu\,dV = \int_{\partial \Omega}\nu\cdot n\,dA,
\]
where $n$ is the unit  normal to $\partial \Omega$ that points out of $\Omega$.
\end{theorem}

Here $dV$ and $dA$ indicate integration with respect to 
$(m+1)$-dimensional volume and $m$-dimensional area (i.e., 
with respect to Hausdorff measure of those dimensions.)

\begin{proof}
Let $\eps_n$ be a sequence of positive numbers converging to $0$.
Note that for each $n$, we can cover $K$ by open balls such that the sum of the areas of the boundaries of the balls
 is less than $\eps_n$. 
 Since $K$ is compact, we can cover it by a finite
collection of such balls.  Let $W(n)$ be the union of those balls.
Note that for each $n$, the balls may be chosen to have arbitrarily small
radii.  In particular, we can choose the balls at stage $n$ so that
$W(n)\subset W(n-1)$.  Note that $\cap_nW(n)=K$.
 Applying the divergence theorem
to $\Omega\setminus W(n)$ gives
\begin{align*}
  \int_{\Omega\setminus W(n)}\Div \nu\,dV
  &=
  \int_{\partial (\Omega\setminus W(n))} \nu\cdot n\,dA
  \\
  &=
  \int_{(\partial \Omega)\setminus W(n)}\nu\cdot n\,dA
  +
  \int_{\Omega\cap\partial W(n)} \nu\cdot n\,dA
\end{align*}
This last integral is bounded in absolute value by the supremum of $|\nu|$
times the area of $\partial W(n)$; that area is bounded by $\eps_n$ by choice of $W(n)$.
Thus
\[
\int_{\Omega\setminus W(n)} \Div\nu\,dV
=
 \int_{(\partial \Omega)\setminus W(n)}\nu\cdot n\,dA + O(\eps_n).
\]
Now use the dominated convergence theorem to take the limit as $n\to\infty$.
(Recall that the $W(n)$ are nested and 
that $\cap_n W(n)=K$.)
\end{proof}

\section{Minimal graphs}
\label{MinimalGraphs}
Let $N$ be a smooth $2$-manifold with boundary.
Let $g$ be a smooth Riemannian metric (not necessarily complete) on $N\times \RR$ that is invariant
under vertical translations. Suppose also that $N\times \RR$ is 
 strictly mean convex, i.e. the mean curvature vector
of $\partial (N\times \RR)$ is a positive multiple of the inward-pointing unit normal. We will assume when necessary that $N $ has been isometrically embedded in some Euclidean space $\RR^k$, so that $N\times \RR$ lies in $\RR ^{k+1}$. Thus translation and dilation make sense. 

We are interested in  graphs over $N$:
Let
\[
  f: N\setminus\partial N\to \RR
\]
be a smooth function whose graph is a $g$-minimal surface $M$. {\em We will assume throughout this Appendix that $M$ is such a graph, and that $\overline{M}$ is a smoothly embedded
manifold-with-boundary, where the boundary is $\Gamma:=\overline{M}\cap\partial N$.} 
We will also assume that
the curvature of $\Gamma$ and its derivative with respect to arclength 
 are bounded above by some $\kappa< \infty$. 
 
Letting $N$ be a convex domain in $\RR^2$ gives the simplest example of this setting. Here, the metric $g$ on $N\times\RR$ is the product metric, and it is the standard Euclidean metric. 
 In this paper we are considering rotationally symmetric (around the $z$-axis $Z$) domains $W\subset \RR ^3$ endowed with  Riemannian metrics that are also rotationally symmetric. 
The simply connected covering space  of $W\setminus Z$ can be written in the form $N\times \RR$, where
 $$N=\{(x,z)\,:\, x>0,\, (x,0,z)\in W\}.$$
 Vertical translations in $N\times \RR$ correspond to rotations in $W\setminus Z$. The metric $g$ on $N\times \RR$ lifted from the metric on $W\setminus Z$ is translation invariant but it is not the product metric.

  \begin{proposition}\label{areaestimate0} 
Suppose that $K$ is any compact region in $(N\setminus\partial N)\times\RR$
 with piecewise-smooth, mean convex boundary.
Then
\begin{enumerate}
\item[1.]The surface
 $M\cap K$ has less area than any other surface in $K$ having the same boundary. 
 \item[2.] Furthermore, 
\begin{equation*}
  \area(M\cap K)\le \frac12\area(\partial K).
\end{equation*}
\end{enumerate}
\end{proposition}

 \begin{theorem}\label{curvatureestimate0}
If $U$ is an open subset of $N$  with compact closure and if $\,\UU = U\times \RR$, then
\begin{equation*}
   B(M,p) \, \dist(p, \UU^c) < C.
\end{equation*}
for some constant $C=C(N,g,\UU,\kappa)<\infty$.
\end{theorem}

\begin{proof}[Proof of Proposition~\ref{areaestimate0}] 
To prove Statement~1, define
\begin{align*} 
&F: N\times \RR \to \RR \\
&F(x,z) = z - f(x).
\end{align*}
Note that the level sets of $F$ are vertical translates of $M$, that these level sets 
  foliate $N\times\RR$, and that $M$ is the level set $F=0$.
Now let $S$ be the least-area surface (flat chain mod $2$) in $K$ having the same boundary
as $M\cap K$.  Then $S$ is smooth except possibly at its boundary. 
Assume that $S$ is not $M\cap K$.
Then $F$ is nonzero, say positive, at some point of $S$.  Let $q$ be the point in $S$ at which 
   $F$ is a maximum.
Since $F=0$ on $\partial S$, $q$ is an interior point.  Thus $S$ lies below the minimal surface $F=F(q)$ but touches it at $q$.
By the maximum principle, the entire connected component  $S'$ of $S$ that contains $q$ must lie in the level set $F=F(q)$. 
   Note that $S'$ must have boundary points, since otherwise $S\setminus S'$ would have the same boundary as $S$ but less area.
However, $F=0$ on $\partial S$, a contradiction.   This completes the proof of Statement~1.

To prove Statement~2, note that
 that $(\partial K)\cap \{F<0\}$ and $(\partial K)\cap \{F\ge 0\}$ both have the same boundary as $M\cap K$,
and their areas add up to $\area(\partial K)$.  Thus
\begin{align*}
 \area(M\cap K) 
 &\le \min\{ \area((\partial K)\cap \{F<0\}), \area((\partial K)\cap\{F\ge 0\} ) 
 \\
 &\le \frac12\area(\partial K). 
\end{align*}
\end{proof}

\begin{proof}[Proof of Theorem~\ref{curvatureestimate0}]
Suppose that the theorem is false.   Then there is a sequence of examples $M_n$
satisfying the hypotheses of the theorem such that
\begin{equation}\label{blowup0}
   \sup_{p\in M_n\cap\,\, \overline{\UU}} B(M_n, p)\, \dist(p, \UU^c) \to \infty.
\end{equation}
Since $M_n\cap\,\, \overline{\UU}$ is compact (it is the graph of a smooth function over $\overline{U}$), the supremum
in \eqref{blowup0} is attained at some point $p_n\in M_n\cap\,\,\overline{\UU}$.   Thus:
\begin{equation}\label{choice}
 B(M_n, p) \, \dist(p, \UU^c) \le 
  B(M_n, p_n)\, \dist(p_n, \UU^c) \to \infty
  \end{equation}
for any  $p\in M_n\cap\,\,\overline{\UU}$.  By vertically translating each $M_n$, we may assume that the height of $p_n$ is $0$.
The assumption about bounds on the curvature of $\Gamma_n=\partial M_n$ imply that we can assume, by passing to 
a subsequence, that the $\Gamma_n$ converge in $C^{2,\alpha}$ to an embedded 
curve $\Gamma$.
(If $\partial N$ is connected, then of course each $\Gamma_n$ is connected.
But $\Gamma$ need not be connected because portions of $\Gamma_n$ may go off to infinity.)

Now translate $M_n$, $\UU$, and $N\times \RR$ by $-p_n$ and dilate by $B(M_n,p_n)$
to get $M_n'$, $\UU_n'$, and $N'\times \RR$.  Note that
\begin{equation}\label{B=1}
B(M_n', 0) = 1
\end{equation}
and, using \eqref{B=1}, the scale invariance of the product  $B(M_n, p)\, \dist(p, \UU^c)$,  and \eqref{choice}, 
\begin{align*}
\dist(0, (\UU_n')^c) &= B(M'_n, 0)\, \dist(0, (\UU_n')^c) 
\\
&=B(M_n, p_n)\, \dist(p_n, (\UU_n)^c) 
\\
&\to \infty.
\end{align*}
In particular, 
\begin{equation}\label{blownaway}
  \dist(0, (\UU_n')^c) \to \infty.
\end{equation}
Choose  $R<\dist(0,\partial \UU_n')$, and let $p$ be a point in $M'_n$ satisfying
$\dist(p,0)<R.$ 
(Here we are using the rescaled metric associated at the $n$th stage.)  For such a choice of $p$ we have from   \eqref{choice},   scale invariance and \eqref{B=1} and the triangle inequality:
\begin{equation*} \label{rescaled0}
\begin{aligned}
   B(M_n', p) 
   &\le
   B(M_n',0) \frac{\dist(0, (\UU_n')^c)}{\dist(p, (\UU_n')^c)} 
   \\
   &=
   \frac{\dist(0, (\UU_n')^c)}{\dist(p, (\UU_n')^c)}  \\
   &\le
   \frac{\dist(0, (\UU_n')^c)}{\dist(0, (\UU_n')^c) - \dist(p, 0) }.  
\end{aligned}
\end{equation*}
Now choose any  fixed $R>0$. By \eqref{blownaway}, we have $R<\dist(0,\partial \UU_n')$ for $n$ large enough.
Choose $p$ so that $\dist(p,\partial \UU_n')<R$. Then from the estimate above
\begin{equation*} \label{rescaled1}
\begin{aligned}
   B(M_n', p) 
   &\le
    \left( 1 - \frac{\dist(p,  0)}{\dist(0, (\UU_n')^c )} \right)^{-1} \\
   &\le 
   \left( 1 - \frac{R}{\dist(0, (\UU_n')^c )} \right)^{-1}.
\end{aligned}
\end{equation*}

This estimate is valid for any $R>0$ and $n$ sufficiently large.
 It follows  that
\begin{equation}\label{limsup}
   \limsup_{n\to\infty} \,( \sup\{B(M_n',p)\,:\, p\in M_n',\,\, \dist(p,0)<R\} ) \le 1.
\end{equation}

 Note that the dilation  factors $B(M_n, p_n)$ are diverging. Hence the metrics $B(M_n, p_n)g$ are becoming the flat metric. 
 The curvature estimate~\eqref{limsup} implies that (after passing to a  subsequence) the $M_n'$ converge
 smoothly to an area-minimizing surface $M'$ in  a flat Euclidean space $E$. Whether $E$ is all of $\RR ^3$ or not depends on what happens as $n\to\infty$ to $\partial N_n'\times\RR$. If $\dist(0, \partial N_n\times\RR)\to\infty$, then $E$ is  Euclidean three-space. If these distances are bounded, then $E$ is a flat halfspace bounded by   a plane corresponding to the limit (after passing to a further subsequence) of the boundaries
 $N_n'\times \RR$.    In the latter case, $\partial M' $ is a straight line lying in the plane
 $\partial E$.
  In either case, from  \eqref{limsup}  and \eqref{B=1}, we can assert that
 \begin{equation}\label{nonflat}
   \sup B(M', \cdot) = B(M', 0) = 1.
 \end{equation}

 \begin{claim} $E$ is  a halfspace,  and $M'\subset E$   is a properly embedded, simply connected area-minimizing minimal surface with quadratic area growth,   whose boundary  $\partial M'$ is a line in the plane $\partial E$.   \end{claim}
 
\begin{proof}[Proof of Claim.]  Each $M_n'$ is a graph. Hence $M'$ is simply connected and properly embedded in $E$. Recall that $M_n$ is stable in $N\times \RR$. Hence $M_n'$ is stable in $N_n'\times\RR$. Stability gives us the estimate
\[
   B(M_n',0) \dist(0, \Gamma_n' \cup (\UU_n')^c)  < c_0
\]
for some constant $c_0$ independent of $n$. Therefore, 
\[
   \dist(0, \Gamma_n' \cup (\UU_n')^c) < c_0
 \]
 since $B(M'_n,0)=1$.  Thus by~\eqref{blownaway},
 \[
  \limsup_{n\to\infty} \dist(0, \Gamma_n') < c.
 \]
It follows that (after passing to a subsequence) the  $\Gamma_n'$ converge smoothly to
 a straight line $\Gamma'$ and that $\partial N_n'\times \RR$ converges smoothly to a limit $E$ that
 is isometric to a closed halfspace of $\RR^3$.  The  boundary  of $E$ contains the line $\Gamma'$.

Observe that if $q\in \partial E$, then it follows from Statement~2 of Proposition~\ref{areaestimate0}
  \[
    \area(M' \cap \BB(q,r)) \le \frac12 \area( \partial (\BB(q,r)\cap E)) = 3\pi r^2.
\]
Thus $M'$ has quadratic area growth.   It follows from Statement~1 of that same proposition that $M'$ is area minimizing.

There are several ways to see that $M'$ must be halfplane, contradicting the fact that
$B(M',0)=1$, which was established in\eqref{nonflat}. Here is one way.  A properly embedded, area-minimizing minimal surface with quadratic area growth
 must be a halfplane or half of Enneper's surface. This was conjectured
 by one of us (White,\cite{white-enneper}) and proved by P\'erez \cite{Perez}. (Here, area-minimizing is used in the classical sense.  
 That is, the allowed comparison surfaces are obtained by compactly supported deformations that vanish on the boundary.)
 According to the Claim above, $M'$ satisfies all the hypotheses, so it must be either a halfplane or half of Enneper's surface.  But $M'$ lies in a halfspace, and half of Enneper's surface does not.  So $M'$ is a halfplane.

 Here is another way to see that $M'$ is a halfplane. Double $M'$ by  Schwartz reflection about its boundary line to produce a complete, simply connected, embedded, minimal surface. As established in the Claim above, $M'$, has quadratic area growth in $\RR^3$, so the same is true for its double. But finite topology  together with quadratic area growth was shown by  P. Li \cite{Li}(see Proposition~32 in \cite{white-lectures})  to imply finite total curvature, and it is well known that the only complete, simply connected, embedded minimal surface of finite total curvature is the plane. 
\end{proof} \end{proof}
\marginpar{End proof of theorem?  Explain why it follows from the claim?}
%{Two endproofs here: one for the claim, one for the theorem.}

\section{Hyperbolic catenoids} \label{hyperbolic-ribbons}
\newcommand{\Cone}{\operatorname{Cone}}
\newcommand{\cone}{\operatorname{cone}}
Consider the hyperbolic metric on the upper halfspace:
\[
 \frac{dx^2 + dy^2 + dz^2}{z^2}.
\]
Let $r=\sqrt{x^2+y^2}$ and $R=\sqrt{x^2+y^2+z^2}$.
Let $\theta$ be the angle that the vector $(x,y,z)$ makes with the horizontal:
\[
  \theta = \arcsin \frac{z}{R} \in [0,\pi/2].
\]

The hemispheres $\{R=\textnormal{constant}\}$ are totally geodesic surfaces of revolution about 
$Z=\{x=y=0\}=\{\theta=\pi/2\}$.  For $\alpha\in (0,\pi/2)$,  the surfaces  
\begin{equation}\label{cone}
{\rm cone}(\alpha):=\{\theta=\alpha\}
\end{equation}
 are surfaces of revolution about $Z$ orthogonal to  the hemispheres.   
 
  The hyperbolic distance $s$ from $Z$ of a point $(x,y,z)\in {\rm cone}(\alpha)$
to $Z$ is given by the following, where $R^2=x^2+y^2+z^2$:
\begin{equation} \label{cone-distance}
s=\int_\alpha^{\pi/2}\frac{Rd\theta}{z} =\int_\alpha^{\pi/2}\frac{d\theta}{\sin\theta} =|\ln\tan(\alpha/2)|,
\end{equation}
From this, we see that ${\rm cone} (\alpha)$ is the set of points at  
constant hyperbolic distance $s$ from $Z$.

If we define $\Cyl(s)$ to be the points at hyperbolic distance equal to or  less than $s$ from $Z$, then
\begin{equation} \label{cones-cylinders}
\begin{aligned}
    \Cyl(s)  &= \{(x,y,z): \pi/2>\theta(x,y,z) \ge \alpha\}, 
    \\
    \partial \Cyl(s) &=  {\rm cone}(\alpha),
\end{aligned}
\end{equation}
where $\alpha$ and $s$ are related by \eqref{cone-distance}. 

 In general, 
consider a surface $\Sigma$ of revolution about $Z$. It can be expressed as
\begin{align*}
R=R(\tau),\quad
\theta=\theta(\tau),\quad
\tau\in I
\end{align*}
where $I\subset \RR$ is some interval.
Since the Euclidean distance to $Z$ is $R\sin\theta$, the
 Euclidean area of an infinitesimal ribbon of $\Sigma$
is given by 
\[
  2\pi R \cos\theta \sqrt{dR^2 + R^2\,d\theta^2}.
\]
Therefore the hyperbolic area of that ribbon is
\begin{align*}
 \frac{ 2\pi R\cos\theta \sqrt{dR^2 + R^2\,d\theta^2} }{z^2} 
 &=
 \frac{2\pi R^2\cos\theta}{z^2} \sqrt{(dR/R)^2 + d\theta^2}
 \\
 &= \frac{2\pi \cos\theta}{\sin^2\theta} \sqrt{ dt^2 + d\theta^2},
\end{align*}
where $t = \log R$ (so $R=e^t$).  Here we have used $z=R\sin\theta$.

Consequently, we see that a surface rotationally invariant about $Z$
is a minimal surface if and only if the corresponding curve in
\[
  \{ (\theta,t)\in (0,\pi/2] \times \RR \}
\]
is a geodesic with respect to the metric
\begin{equation} \label{eq:metric}
   \frac{2\pi \cos\theta}{\sin^2\theta} \sqrt{ dt^2 + d\theta^2}.
\end{equation}

Now suppose we have a geodesic given by
\[
  t = t(\theta), \quad \theta\in I
\]
where $I\subset (0,\pi/2]$ is an interval.  Then the length is
\[
   \int_{\theta\in I}  \frac{2\pi \cos\theta}{\sin^2\theta} \sqrt{ dt^2 + d\theta^2}
   =
   \int_{\theta\in I}  \frac{2\pi \cos\theta}{\sin^2\theta} \sqrt{ t'(\theta)^2 + 1}\,d\theta.
\]
Since the integrand does not depend on $t$, the Euler-Lagrange equation
for this functional (i.e., the equation for a geodesic) is
\[
\frac{d}{d\theta}
\left(
 \frac{2\pi \cos\theta}{\sin^2\theta} \frac{t'(\theta)}{\sqrt{ t'(\theta)^2 + 1}}
\right) = 0
\]
or
\begin{equation}
\frac{2\pi \cos\theta}{\sin^2\theta} \frac{t'(\theta)}{\sqrt{ t'(\theta)^2 + 1}} = c \label{E-L-eqn}
\end{equation}
for some constant $c$.

 From \eqref{E-L-eqn} we have the following result:

\begin{theorem}\label{very-close-theorem}
For $\theta$ near $0$, 

\begin{equation}\label{t-prime-estimate}
   t'(\theta) = O(\theta^2)
\end{equation}
and therefore 
\begin{equation}
  t(\theta)-t(0) = O(\theta^3).
\end{equation}
\end{theorem}

\begin{corollary} \label{length-I-theta}
Consider two geodesics in $(0,\pi/2]\times\RR$ converging to the same ideal
boundary point.  The vertical distance between them tends to $0$ as $\theta\to 0$.
That is, if $t_1(\cdot)$ and $t_2(\cdot)$ are two solutions of the Euler-Lagrange
equation with $t_1(0)=t_2(0)$, and if $I(\theta)$ is the vertical segment joining
$(\theta, t_1(\theta))$ and $(\theta, t_2(\theta))$, then the length of $I(\theta)$ 
(with respect to the metric~\eqref{eq:metric})
 is $O(\theta)$
as $\theta\to 0$.
\end{corollary}

\newcommand{\length}{\operatorname{length}}
\begin{proof}
We can let $t_1(\cdot)$ be any solution $t(\cdot)$, and we 
may as well take $t_2(\theta)$ to be the horizontal geodesic $t_2(\theta)\equiv t(0)$.
Now
\begin{align*}
\length(I(\theta))
&=
\frac{2\pi \cos\theta}{\sin^2\theta} \length_\textnormal{eucl}I(\theta) 
\\
&= 
\frac{2\pi \cos\theta}{\sin^2\theta} | t(\theta)-t(0)|
\\
&=
\frac{2\pi \cos\theta}{\sin^2\theta} O(\theta^3),
\end{align*}
which is clearly $O(\theta)$.
\end{proof}
\begin{remark}\label{ribbon-theorem-proof} The length of $I(\theta)$ equals the area of the ribbon  on 
$\cone(\theta)$ between the rotational minimal surfaces that correspond to the two
geodesics converging to the same ideal-boundary point. By \eqref{cones-cylinders},
$\cone(\theta)=\partial\Cyl(s)$, and by   \eqref{cone-distance},  $s\rightarrow\infty$ if and only if   $\theta\to 0$. Therefore, Theorem~\ref{ribbon-theorem}  follows from Corollary~\ref{length-I-theta}.
 \end{remark}

We now compute the curvature of the Riemannian metric~\eqref{eq:metric}.

\begin{lemma}\label{Curvature}
Let $\lambda=\lambda(\theta,t)=2\pi \cos\theta/ \sin^2\theta$.
The Gauss curvature $K$ of the metric 
 $\lambda\sqrt{d\theta^2+dt^2}$ on the strip $(\theta,t)\in (0,\pi/2)\times\RR$ is
 given by
 \[
   K = K(\theta) = \frac1{4\pi^2}\,\tan^2\theta \left( \tan^2\theta - 2 \right).
 \]
 In particular, $K\ge 0$ if and only if $\theta\ge \alpha_0:=\arctan\sqrt2$.
 \end{lemma}
 
\begin{proof}  We use the following formula for the Gauss curvature of a surface with a conformal metric 
$\lambda \sqrt{d\theta^2+dt^2}$:
\begin{equation}\label{K}
K=\frac{-\Delta\ln \lambda}{\lambda^2}.
\end{equation} 
We compute
 \begin{align*}
(\ln \lambda)'  &= -\tan\theta + 2\cot\theta \\
(\ln \lambda)'' &= -\frac1{\cos^2\theta} + \frac2{\sin^2\theta}.
\end{align*}
Thus
\[
K 
= 
\frac{\sin^4\theta}{4\pi^2 \cos^2\theta} \left( \frac1{\cos^2\theta}  - \frac2{\sin^2\theta} \right)
=
\frac1{4\pi^2} \tan^2\theta \left( \tan^2\theta - 2 \right).
\]
\end{proof}

\begin{proposition} Let $C$ and $C'$ be minimal annuli of rotation  
with a common axis $Z$ in hyperbolic thee-space.
Suppose that both of these annuli lie outside the cylinder ${\bf Cyl}( \ln \tan(\alpha_0/2))$,
as defined in~\eqref{cones-cylinders}. Then
 Then $C$ and $C'$ can intersect in at most one circle. 
 In particular, no two such annuli have the same boundary.
 \end{proposition}

Here,  $\alpha_0=\arctan\sqrt{2}$ as in Lemma~\ref{Curvature} above. The proposition follows immediately from  Lemma~\ref{Curvature}  and the observation that on a surface of negative curvature, two distinct geodesics cannot cross more than once, a simple consequence of the Gauss-Bonnet formula. (By construction, geodesics in the strip
$(0,\pi/2)\times \RR$ correspond to minimal annuli of rotation in hyperbolic three-space.)

%\nocite{pedrosa-ritore}
%\nocite{hoffman-wei}
\newcommand{\hide}[1]{}
%\bibdata

\end{document}